\documentclass{amsart}

\newif\ifsolutions
\title{Asymptotics for Bailey-type mock theta functions}
\author{Taylor Garnowski}

\solutionstrue   

\usepackage{subcaption}
\usepackage{comment}
\usepackage{float}
\usepackage[dvipsnames]{xcolor}
\usepackage{graphicx}
\usepackage[ backend = biber, style = numeric]{biblatex}
\renewbibmacro{in:}{}
\usepackage{mathtools}
\DeclarePairedDelimiter\floor{\lfloor}{\rfloor}
\usepackage{authblk}
\usepackage[all,knot]{xy}
\usepackage{tabularx}
\usepackage{setspace}

\usepackage[utf8]{inputenc}
\usepackage{amsthm,amssymb,amstext,amscd,amsfonts,amsbsy,latexsym,amsmath,xcolor,mathrsfs,fancybox,upgreek, soul,url}
\usepackage[english]{babel}

\usepackage{cancel}
\usepackage[draft]{hyperref}

\usepackage{comment}
\usepackage{mdframed}
\allowdisplaybreaks
\usepackage{enumerate}
\usepackage{thmtools}
\usepackage{thm-restate}
\usepackage{chngcntr}
\usepackage{enumitem}
\usepackage{etoolbox}
\usepackage{todonotes}
\usepackage{subcaption}

\newtheorem{Theorem}{Theorem}
\newtheorem{Lemma}[Theorem]{Lemma}
\newtheorem{Example}[Theorem]{Example}
\newtheorem{Corollary}[Theorem]{Corollary}
\newtheorem{Remark}[Theorem]{Remark}
\newtheorem{Conjecture}{\bf Conjecture}
\newtheorem{Proposition}[Theorem]{Proposition}
\newtheorem{Definition}[Theorem]{Definition}
\theoremstyle{definition}

\date{Received: date / Accepted: date}

\begin{document}

\date{}
\maketitle

\begin{abstract}
   We compute asymptotic estimates for the Fourier coefficients of two mock theta functions, which come from Bailey pairs derived by Lovejoy and Osburn. To do so, we employ the circle method due to Wright and a modified Tauberian theorem. We encounter cancellation in our estimates for one of the mock theta functions due to the auxiliary function $\theta_{n,p}$ arising from the splitting of Hickerson and Mortenson. We deal with this by using higher order asymptotic expansions for the Jacobi theta functions. 
 \end{abstract}

\section{Introduction}
\subsection{History}
\hspace{5mm} We recall the defintion of a classical mock theta function. Let $q$ be a complex variable with $|q|<1$. A classical mock theta function $M(q)$ is a function for which near each root of unity $\xi$ there exists a weakly holomorphic modular form, $F_{\xi}$, and a rational number $a_\xi$, such that near $\xi$
\begin{align}\label{mockdef}
    M(q) - q^{\alpha_\xi}F_{\xi}(q) = O(1).
\end{align}
We then eliminate the possibility of having holomorphic theta functions from the definition by declaring that no $F_{\xi}$ satisfies the above condition for all roots of unity. A nice list of the classical mock theta functions exists in the appendix of \cite{Bringmannbook} and Section 4 of \cite{splittingmort}. Large families of new examples of modular type functions that satisfy Eq. \eqref{mockdef} were discovered after S. Zwegers wrote his thesis \cite{Zwegers} on mock theta functions in 2002, whereby the classical mock theta functions were found to be linked to harmonic Maass forms. As a result, functions that are finite sums of normalized Appell sums can be viewed as mock theta functions. This result brings the theory of mock theta functions and combinatorial generating functions closer together. For example, let $\zeta := e^{2\pi i z}$, then the famous partition rank generating function,
\begin{align*}
    R(z;\tau):=\sum^\infty_{n=0}\sum_{m\in \mathbb{Z}} N(m,n)\zeta^m q^n:= \sum_{n\geq 0}\frac{q^{n^2}}{(\zeta q,\;\zeta^{-1}q;q)_n},
\end{align*}
can be written as a sum of normalized Appell sums and is thus a mock theta function when $z\in \mathbb{H}$ (Lemma 3.1 in \cite{DousseMertensAsymp}). Understanding how the coefficients of mock theta functions grow is important, especially when a combinatorial interpretation is available. For example (see Theorem 1.2 in \cite{DousseMertensAsymp}), 
\begin{align*}
    N(m,n) \sim \frac{\beta}{4}\textnormal{sech}\left(\frac{\beta m}{2}\right)p(n) =  \frac{\beta}{16\sqrt{3}n}\textnormal{sech}\left(\frac{\beta m}{2}\right)e^{\pi\sqrt{\frac{2n}{3}}},
\end{align*}
where $p(n)$ is the partition function, $\beta:= \frac{\sqrt{n}\textnormal{log}(n)}{\pi\sqrt{6}}$, and $f(n) \sim g(n)$ denotes that the ratio of $f(n)$ and $g(n)$ goes to $1$ as $n \to \infty$.
\subsection{Bailey pairs and Mock theta functions}
\hspace{5mm} The inspiration for this work comes from the fact that we want to find similar asymptotic estimates for mock theta functions that come from $\textit{Bailey pairs}$. Let $\alpha_n(q)=:\alpha_n $ and $\beta_n(q)=: \beta_n$ be two sequences of $q$-series. The tuple $(\alpha_n,\beta_n)$ is referred to as a Bailey pair with respect to $a\in \mathbb{C}$ (assuming $a$ causes no poles in what follows) if 
\begin{align}\label{baileyrep1}
    \beta_n = \sum^n_{k=0}\frac{\alpha_k}{(q)_{n-k}(aq)_{n+k}}.
\end{align}
 The fact that Bailey pairs and mock theta functions are related is not immediately obvious, and it wasn't until  Andrews  showed that Eq. \eqref{baileyrep1} can be iterated to obtain an infinite family of Bailey pairs that a true connection was found \cite{andrewsbaileychain1,andrewsbaileychain2}. This is the content of \textit{Bailey's lemma} \cite{andrewsbaileychain1,andrewsbaileychain2,bailey2,bailey1}. Bailey's lemma leads to  families of sums, known as \textit{higher level Appell sums}, which are not necessarily mock theta functions, but  mixed mock theta function \cite{Bringmannbook,lovejoybaileypairs}. Occasionally, certain pairs lead to  normal Appell sums via Bailey's lemma, and we call the resulting functions \textit{Bailey-type mock theta functions}.
\par The study of Bailey-type mock theta functions became more interesting with a key result by Hickerson and Mortenson  \cite{splittingmort}, which gave an explicit decomposition of indefinite theta functions in terms of Appell sums and theta functions. This result was used by many authors in works such as \cite{Gupaper,loveindefinite,lovejoybaileypairs} to write families of Bailey-type mock theta functions in terms of classical mock theta functions. For example, Lovejoy and Osburn in \cite{lovejoybaileypairs} derived a Bailey-type mock theta function, $R^{(4)}_1(q)$, and used the decomposition of \cite{splittingmort} to find the formula
\begin{align}\label{R41}
    R^{(4)}_1(q) = -\phi(q^4) + M_1(q),
\end{align}
where $\phi$ is the 10th order classical mock theta function given by
\begin{align*}
    \phi(q):= \sum^{\infty}_{n=0}\frac{q^{\frac{n(n+1)}{2}}}{(q;q^2)_{n+1}}
\end{align*}
and $M_1(q)$ is a weakly holomorphic modular form. Understanding how the coefficients of certain Bailey-type mock theta functions grow is an interesting question, which was proposed by Lovejoy and Osburn in \cite{lovejoybaileypairs}, and which we will begin to answer in this work. To the best of our knowledge, no works have investigated the growth of Bailey-type mock theta functions in depth. Doing so here for two example functions, we hope to lay the groundwork for future and more advanced studies of the asymptotic properties of Bailey-type mock theta functions. Let $a(n)$ denote the coefficients of $R^{(3)}_3$ and $b(n)$ the coefficients of $R^{(3)}_1$, which are two Bailey-type mock theta functions defined in Definition \ref{mainpairs} (the $a(n)$ and $b(n)$ are explicitly defined in Examples 1 and 2) . We will show the following. 
\begin{Theorem}\label{thebigboyforaandb}
The following estimates hold as $n \to \infty$\textnormal{:}
\begin{align*}
    a(n) &\sim  (-1)^n\frac{\sqrt{6}}{12\sqrt{n}}e^{\pi \sqrt{\frac{n}{12}}},\\
    b(n) &\sim \left(\frac{1}{2\textnormal{sin}\left(\frac{\pi}{4}\right)\textnormal{sin}\left(\frac{5\pi}{12}\right)} +1\right)\frac{e^{\pi\sqrt{\frac{n}{6}}}}{\sqrt{24n}}.
\end{align*}
\end{Theorem}

\par The following table shows the ratio between the estimated values in Theorem \ref{thebigboyforaandb} and the actual values for some values of $n$.
\vspace{5mm}

\begin{tabularx}{0.8\textwidth} { 
  | >{\raggedright\arraybackslash}X 
  | >{\centering\arraybackslash}X 
  | >{\raggedleft\arraybackslash}X | }
 \hline
 $n$ & $a(n)$  & $b(n)$ \\
 \hline
 $100$  & $0.96315$  & $0.98067$  \\
\hline
$500$  & $0.98249$  & $0.99081$  \\
\hline
$1000$  & $0.98740$  & $0.99343$ \\
\hline
\end{tabularx}

\vspace{5mm}

\par
 To obtain  asymptotic estimates like the ones we give in  our main Theorem \ref{thebigboyforaandb}, it is often useful to use a modified circle method due to Wright \cite{Wrightcircle}, which allows one to look at a finite number of poles. Wright's technique has been used by several authors in recent years \cite{Karlandkathwright,BringDousseLoveMahl,DousseMertensAsymp,awayfromqrichard} to deal with combinatorial generating functions like $R(z;\tau)$, for example. The common theme here and in the works \cite{Karlandkathwright,BringDousseLoveMahl,DousseMertensAsymp,awayfromqrichard} is that the functions are generically mixed mock modular forms, which are more suited for the adapted circle method of Wright.

\par This work is organized as follows: 
In Section \ref{sectionprelim}, we define the main objects of this work. In Section \ref{near0estsection}, we provide estimates near $\tau =0$ of the Jacobi theta function and normalized Appell sum. In Sections \ref{Theanbaby} and 5, we employ the Wright circle method to prove the first part of our main theorem, and in Section \ref{thebnsec} we use results from \cite{TauberianSchaffer,Ingham} to prove the second part of our theorem. Finally, we offer some remarks on our results and thoughts on future work regarding this topic in Section \ref{conclusion}.

\section*{Acknowledegments}
\hspace{5mm} This work was supervised by Kathrin Bringmann, and we would like to thank her for her contributions. We would like to give special thanks to Caner Nazaroglu for giving insight into many of the calculations in this work and Jeremy Lovejoy for his helpful comments and suggestions regarding many identities. We finally want to thank the anonymous reviewer, Chris Jennings-Schaffer, Alexandru Ciolan, and Markus Schwagenscheidt for their helpful suggestions and edits.

\section{Preliminaries and basic definitions}\label{sectionprelim}
\hspace{5mm} The basic objects that appear in this work, and some of their properties, are collected in this section. We begin by recalling  the definitions of the normalized Appell sum and the Jacobi theta function:
\begin{align}\label{mudef}
     \mu(z_1,z_2;\tau) &:= \frac{\zeta^{\frac{1}{2}}_1}{\vartheta(z_2;\tau)}\sum_{n\in \mathbb{Z}}\frac{(-1)^nq^{\frac{n^2+n}{2}}\zeta^n_2}{1-q^n\zeta_1},
\end{align}
where $z_1,z_2 \in \mathbb{C}$, $\zeta_j:=  e^{2\pi i z_j}$ , $q := e^{2\pi i \tau}$,  $\tau \in \mathbb{H}$, and $\vartheta$ is the \textit{Jacobi theta function} (or $\vartheta$-function, for short) given by
\begin{align}\label{jacthetadef}
    \vartheta(z_2; \tau):= \sum_{m\in \frac{1}{2}+\mathbb{Z}}(-1)^mq^{\frac{m^2}{2}}\zeta^m_2.
\end{align}
Furthermore, we have the \textit{Jacobi product representation} for the $\vartheta$-function:
\begin{align}\label{thetaasjacprod}
\vartheta(z;\tau) = -iq^{\frac{1}{8}}\zeta^{-\frac{1}{2}}(\zeta;q)_{\infty}(q;q)_{\infty}(\zeta^{-1}q;q)_{\infty},
\end{align}
where $\zeta:= e^{2\pi i z}$. Many of the important functions discussed here were originally defined in  \cite{splittingmort,lovejoybaileypairs}. In those works, the authors used a slightly different notation for the $\vartheta$-function (denoted by $j$) and the Appell sum (denoted by $m$). One can go between the two via the formulas
\begin{align*}
    \vartheta(z_2;\tau) &= -iq^{1/8}\zeta_2^{-1/2}j(\zeta_2,q),\\
     m(\zeta_1; q; \zeta_2) &= iq^{1/8}\zeta_1^{-1/2}\mu(z_1+z_2, z_2; \tau).
\end{align*}
\hspace{5mm} We will use the following identities frequently.
\begin{Proposition}[See Ch. 1 of \cite{Zwegers}]\label{transformprop} The normalized Appell sum and $\vartheta$-function satisfy\textnormal{:}
\begin{enumerate}
    \item $\mu(z_1,z_2;\tau +1) = e^{-\frac{\pi i}{4}}\mu(z_1,z_2;\tau) $,
    \item $\mu(z_1+1,z_2;\tau) = \mu(z_1,z_2+1;\tau) = -\mu(z_1,z_2;\tau)$,
    \item\label{mutransformitem} $\mu(\frac{z_1}{\tau},\frac{z_2}{\tau};-\frac{1}{\tau}) = -\sqrt{-i\tau}e^{-\pi i \frac{(z_1-z_2)^2}{\tau}}\mu(z_1,z_2;\tau) + \frac{\sqrt{-i\tau}}{2i}e^{-\pi i \frac{(z_1-z_2)^2}{\tau}}h(z_1-z_2;\tau) $, where $h(z;\tau)$ is the Mordell integral given by
  \begin{align*}\label{mordelldef}
      h(z;\tau):= \int_{\mathbb{R}}\frac{e^{\pi i \tau x^2}e^{-2\pi z x}}{\textnormal{cosh}(\pi x)}dx,
  \end{align*}
    \item $h\left(\frac{z}{\tau}; -\frac{1}{\tau}\right) = \sqrt{-i\tau}e^{-\frac{\pi iz^2}{\tau} }h(z;\tau)$,
    \item $\vartheta(z+\tau;\tau) = -e^{-\pi i \tau-2\pi i z}\vartheta(z,\tau)$,
    \item $\vartheta(z;\tau+1) = e^{\frac{\pi i}{4}}\vartheta(z;\tau)$,
    \item $\vartheta(z+1;\tau)=-\vartheta(z;\tau),$
    \item\label{thetatransitem}
        $\vartheta\left(\frac{z}{\tau}; -\frac{1}{\tau}\right) = -i\sqrt{-i\tau}e^{\frac{\pi i z^2}{\tau}}\vartheta(z; \tau).$
\end{enumerate}
\end{Proposition}

For $k\geq 3$ Lovejoy and Osburn showed that the following family of functions  are mock theta functions.

\begin{Definition}[\cite{lovejoybaileypairs}]\label{mainpairs}
Let $k\geq 3$ and $n_1,...,n_k$ be integers such that $1\leq n_1\leq...\leq n_k$. Define
\begin{align*}
    B_k(n_k,n_{k-1},...,n_1;q):= (-1)^{n_1}(-q)_{n_{k-1}}q^{\tiny{\begin{pmatrix}n_{k-1}+1\\2\end{pmatrix}}}\frac{\prod^{k-1}_{j=2}q^{2^{j-2}n_{k-j}}\left(-q^{2^{j-2}};q^{2^{j-2}} \right)_{2n_{k-j}}}{\prod^k_{j=1}\left( q^{2^{j-1}};q^{2^{j-1}}\right)_{n_{k-j+1}-n_{k-j}}},
\end{align*}
with $n_0 := 0$. Then we define
\begin{align*}
    R^{(k)}_1(q) &:= \sum_{n_k\geq n_{k-1}\geq ...\geq n_1\geq 0}q^{{n_k+1}\choose{2}}B_k(n_k,...,n_1;q),\\
    R^{(k)}_3(q) &:= \sum_{n_k\geq n_{k-1}\geq ...\geq n_1\geq 0}\frac{(-1)^{n_k}q^{n^2_k+2n_k}(q;q^2)_{n_k}}{(-q^2;q^2)_{n_k}}B_k\left(n_k,...,n_1;q^2\right).
\end{align*}
\end{Definition}

The authors of \cite{lovejoybaileypairs} showed that
\begin{equation}\label{Rdef}
    R^{(3)}_1(q)= \nu(-q),
\end{equation}
where $\nu(q):=\sum_{n \geq 0}\frac{q^{n^2+n}}{(-q;q^2)_{n+1}}$ is a classical third order mock theta function. 
\vspace{5mm}

The first definition comes from the work of \cite{splittingmort}, and uses the standard combinatorial notation for the Jacobi triple product
\begin{align*}
    j(x,q):= (x)_{\infty}\left(qx^{-1}\right)_{\infty}(q)_{\infty},
\end{align*}
where $x$ is a non-zero complex number. When $x$ is an integral or half integral power of $q$, we will always write $j$ in terms of a $\vartheta$-function as discussed in Section \ref{sectionprelim}, via the transformations
\begin{align*}
    \vartheta(a\tau;b\tau) &= -i q^{\frac{b}{8}}q^{-\frac{a}{2}}j(q^a,q^b),\\
    \vartheta\left(a\tau +\frac{1}{2}; b\tau\right) &= -q^{\frac{b}{8}}q^{-\frac{a}{2}}j(-q^a,q^b).
\end{align*}

\begin{Definition}[see Section 2, \cite{lovejoybaileypairs} and Theorem 1.3, \cite{splittingmort}]\label{thetasplitdef} Let $x$ and $y$ be complex numbers so that they do  not cause poles in the quotients that follow. Then for positive integers $n,p$, $r := r^* + \left\{\frac{(n-1)}{2}\right\}$ and $s:= s^* + \left\{\frac{(n-1)}{2}\right\}$, with $\{a\}$ denoting the fractional part of the number $a$,  define the function $\theta_{n,p}(x,y,q)$ by,
\begin{align*}
 &\theta_{n,p}(x,y,q):=\\ &\frac{j^3\left(q^{p^2(2n+p)},q^{3p^2(2n+p)}\right)}{j\left(-1,q^{np(2n+p)}\right)}\;\Bigg\{\sum^{p-1}_{r^* = 0}\sum^{p-1}_{s^* = 0} q^{n\begin{pmatrix} r - \frac{(n - 1)}{2}\\ 2\end{pmatrix}+(n + p)\left(r - \frac{(n - 1)}{2}\right)\left(s + \frac{(n + 1)}{2}+ n\begin{pmatrix}s + \frac{(n + 1)}{2}\\ 2\end{pmatrix}\right)} \\
 &\hspace{50mm}\times(-x)^{r - \frac{n-1}{2}}(-y)^{s + \frac{(n + 1)}{2}} \\
 &\times \frac{j\left(-q^{pn(s - r)}\frac{x^n}{y^n}, q^{np^2}\right)\;j\left(q^{p(2n + p)(r + s)+p(n + p)}x^py^p, q^{p^2(2n + p)}\right)}{  j\left(q^{pr(2n + p)+\frac{p(n + p)}{2}}\frac{(-y)^{n + p}}{(-x)^n}, q^{p^{2}(2n + p)}\right)\;j\left(q^{ps(2n + p)+\frac{p(n + p)}{2}}\frac{(-y)^{n + p}}{(-x)^n},\; q^{p^{2}(2n + p)}\right)    }\Bigg\},
\end{align*}

where $b\choose c$ is the standard binomial coefficient.
\end{Definition}
Recall that 
\begin{align*}
    m(\zeta_1; q; \zeta_2) &= iq^{1/8}\zeta_1^{-1/2}\mu(z_1+z_2, z_2; \tau).
\end{align*}

We then have the following theorem.
 \begin{Theorem}[\cite{lovejoybaileypairs}] For $k\geq 3$ the function $R^{(k)}_3(q)$ is a mock theta function and satisfies the formula
 \begin{align*}
     R^{(k)}_3(q) &= 2q^{-2^{k-3}\left(2^{k-2}+1\right)}m\left(q^{2^{k-2}},q^{2^{2k-2}+2^k},-1 \right) -2q^{\frac{1}{8}}\frac{\theta_{1,4}\left(q^{2^{k-2}+1},-q^{2^{k-2}+1}, q\right)}{\vartheta\left(\frac{1}{2};\tau\right)}\\
     &=2iq^{-2^{k-3}}\mu\left(2^{k-2}\tau+\frac{1}{2}, \frac{1}{2}; \left(2^{2k-2}+2^k\right)\tau   \right)-2q^{\frac{1}{8}}\frac{\theta_{1,4}\left(q^{2^{k-2}+1},-q^{2^{k-2}+1}, q\right)}{\vartheta\left(\frac{1}{2};\tau\right)}.
 \end{align*}
 \end{Theorem}

 \begin{Example}[The function $R_{3,3}$]\label{R3_3example}
The Fourier expansion for $R_{3,3}(q)$ takes the shape
 \begin{align*}
     & 1-{q}^{3}+{q}^{4}-{q}^{5}+{q}^{6}+{q}^{8}-{q}^{9}\\&+{q}^{10}-2\,{q}^{11
}+2\,{q}^{12}-2\,{q}^{13}+{q}^{14}-2\,{q}^{15}+2\,{q}^{16}-{q}^{17}+3
\,{q}^{18}\\&-3\,{q}^{19}+3\,{q}^{20}-4\,{q}^{21}+3\,{q}^{22}-2\,{q}^{23}
+4\,{q}^{24}-4\,{q}^{25}+4\,{q}^{26}\\&-6\,{q}^{27}+5\,{q}^{28}-6\,{q}^{
29}+6\,{q}^{30}-5\,{q}^{31}+6\,{q}^{32}-6\,{q}^{33}+7\,{q}^{34}\\&-9\,{q}
^{35}+9\,{q}^{36}-9\,{q}^{37}+9\,{q}^{38}-9\,{q}^{39}+11\,{q}^{40}-10
\,{q}^{41}+12\,{q}^{42}\\&-14\,{q}^{43}+13\,{q}^{44}-16\,{q}^{45}+15\,{q}
^{46}-14\,{q}^{47}+17\,{q}^{48}-16\,{q}^{49}+O \left( {q}^{50}
 \right).
 \end{align*}
 One can see the alternating sign changes and our main theorem shows that this behavior holds in the limit $n\to \infty$  We can explicitly get the function $R_{3,3}(q)$ in a form that is suitable for applying the circle method:
 \begin{align*}
     R_{3,3}(q) &= 2iq^{-1}\mu\left(2\tau+\frac{1}{2}, \frac{1}{2}; 24\tau   \right)-2q^{\frac{1}{8}}\frac{\theta_{1,4}\left(q^3,-q^3, q\right)}{\vartheta\left(\frac{1}{2};\tau\right)}\\
     &= 2iq^{-1}\mu\left(2\tau+\frac{1}{2}, \frac{1}{2}; 24\tau   \right)-
     \frac{2q^{\frac{{25}}{8}}}{\vartheta\left(\frac{1}{2}; \tau\right)}\frac{j^3(q^{96},q^{288})}{j(-1,q^{24})}\\
     &\;\;\;\;\;\;\;\;\times\left(\sum^3_{r,s = 0}(-1)^rq^{\frac{r(r-1)}{2}+\frac{s(s+1)}{2}+ 5r(s+1)+3(r+s)}\;\frac{j\left(q^{4(s-r)},q^{16} \right)\;j\left(q^{24(r+s)+44},q^{96} \right)}{j\left(-q^{24r +22},q^{96}\right)\; j\left(-q^{24s +22},q^{96}\right)}\right)\\
     &=  2iq^{-1}\mu\left(2\tau+\frac{1}{2}, \frac{1}{2}; 24\tau   \right) + \frac{2iq^{\frac{417}{8}}}{\vartheta\left(\frac{1}{2}; \tau\right)}\frac{\vartheta^3(96\tau; 288\tau)}{\vartheta\left(\frac{1}{2}; 24\tau\right)}\\
     &\;\;\;\;\;\;\;\;\times\left(\sum^3_{r,s = 0}\Bigg((-1)^rq^{Q(r,s)}\cdot \frac{\vartheta\left(4(s-r)\tau; 16\tau \right)\;\vartheta\left(\{24(r+s)+44\}\tau; 96\tau \right)}{\vartheta\left(\frac{1}{2}+\{24r +22\}\tau ; 96\tau\right)\; \vartheta\left(\frac{1}{2}+\{24s +22\}\tau; 96\tau\right)}\Bigg)\right)\\
     &:= 2iq^{-1}\mu\left(2\tau+\frac{1}{2}, \frac{1}{2}; 24\tau   \right) +T(\tau) ,
 \end{align*}
 where we defined
 \begin{align*}
     Q(r,s):= \frac{r(r-1)}{2}+\frac{s(s+1)}{2}+ 5s+6r+5rs,
 \end{align*}
 above. The study of $T(\tau)$ will be the main focus for the rest of the chapter.
 \end{Example}
 

 \begin{Example}[The function $R^{(3)}_1$]
 Recall that
\begin{align*}
    R^{(3)}_{1}(q) = \nu(-q).
\end{align*}
 With this information, we can show the following formula holds with $\tau \mapsto \tau +\frac{1}{2}$ (see A.2, \cite{Bringmannbook})
    \begin{align*}
         R^{(3)}_{1}(q) &= -2iq^{-\frac{1}{2}}\mu(5\tau,3\tau;12\tau) + e^{\frac{-\pi i}{12}}q^{-\frac{1}{3}}\frac{\eta(\tau + \frac{1}{2})\eta(3\tau+\frac{1}{2})\eta(12\tau)}{\eta(2\tau)\eta(6\tau)}.
    \end{align*}
 \end{Example}

\section{Preliminary estimates for modular theta functions and Appell sums near $\tau =0$}\label{near0estsection}

\hspace{5mm} We collect all of the necessary estimates for the accessory objects that appear in this work near the point $\tau =0$. We have two subcategories of estimates that we need to deal with: the classical estimates that only need one error term, and the  higher order estimates that keep many error terms in the asymptotic expansion.

\subsection{Classical estimates}
\hspace{5mm} We begin with the $\vartheta$-functions near the origin. 
\begin{Lemma}\label{thetatestimateprelim}
Let $\alpha \in [0,1)$, let $q:= e^{2\pi i \tau}$, $q_0 := e^{-\frac{2\pi i }{\tau}}$, and $k > 1$ be a rational number. As $\tau \to 0$ within a cone,
\begin{align}
     \vartheta(\alpha\tau; \tau) &=\label{fulllatticethetadecay} \frac{-2i\; \textnormal{sin}(\pi\alpha)q^{-\frac{\alpha^2}{2}}q^{\frac{1}{8}}_0}{\sqrt{-i\tau}}\left( 1 +  O(q_0) \right) ,\\
     \vartheta\left(\frac{1}{k} + \alpha\tau; \tau\right) &= \label{mixmibabi1moretime} -\frac{q^{-\frac{\alpha^2}{2}}e^{\pi i\alpha(1-\frac{2}{k})}}{\sqrt{-i\tau}}q^{\frac{1}{2k^2}-\frac{1}{2k}+\frac{1}{8}}_0\left(1 + O\left(q^{\frac{1}{k}}_0\right)  \right),\\
     \eta(\tau)&= \label{etagrowth}\frac{q^{\frac{1}{24}}_0}{\sqrt{-i\tau}}\left(1+O(q_0)\right).
\end{align}
\end{Lemma}.
\proof
We begin with Eq. \eqref{fulllatticethetadecay}. Using the Jacobi product formula and Proposition \ref{transformprop}.\ref{thetatransitem} , we have as $\tau \to 0$

\begin{align*}
    \vartheta(\alpha\tau; \tau) &=\frac{iq^{-\frac{\alpha^2}{2}}}{\sqrt{-i\tau}}\vartheta\left(\alpha; -\frac{1}{\tau}\right) =\frac{e^{-\pi i \alpha}q^{-\frac{\alpha^2}{2}} q^{\frac{1}{8}}_0}{\sqrt{-i\tau}}(e^{2\pi i \alpha};q_0)_{\infty} (q_0e^{-2\pi i \alpha};q_0)_{\infty} (q_0;q_0)_{\infty}.
\end{align*}
Thus,
\begin{align*}
   \vartheta(\alpha\tau; \tau) &= \frac{e^{-\pi i \alpha}q^{-\frac{\alpha^2}{2}} q^{\frac{1}{8}}_0}{\sqrt{-i\tau}}\Big(1-e^{2\pi i \alpha} + O(q_0)\Big)\Big(1+O(q_0)\Big)\Big(1+O(q_0)\Big)\\
    &=\frac{e^{-\pi i \alpha}q^{-\frac{\alpha^2}{2}} q^{\frac{1}{8}}_0}{\sqrt{-i\tau}}(1-e^{2\pi i \alpha} + O(q_0)) =\frac{e^{-\pi i \alpha}q^{-\frac{\alpha^2}{2}}( 1 - e^{2\pi i \alpha})q^{\frac{1}{8}}_0}{\sqrt{-i\tau}}\left( 1 +  O(q_0) \right)\\
    &= \frac{-2i \textnormal{sin}(\pi\alpha)q^{-\frac{\alpha^2}{2}}q^{\frac{1}{8}}_0}{\sqrt{-i\tau}}\left( 1 +  O(q_0) \right) ,
\end{align*}
where the second to last step follows from the fact that $1-e^{2\pi i \alpha}$ is $O(1)$. Similarly  for Eq. \eqref{mixmibabi1moretime},
\begin{align*}
    &\vartheta\left(\frac{1}{k} + \alpha\tau; \tau  \right) = \frac{ie^{\frac{-\pi i\left(\frac{1}{k^2}+\frac{2\alpha\tau}{k}+\alpha^2\tau^2\right) }{\tau}}}{\sqrt{-i\tau}}\vartheta\left(\frac{1}{k\tau}+ \alpha ;\frac{-1}{\tau}\right)\\
    &= \frac{ie^{-\frac{2\pi i\alpha}{k}}q^{-\frac{\alpha^2}{2}}q^{\frac{1}{2k^2}}_0}{\sqrt{-i\tau}}\left(-iq^{\frac{1}{8}}_0e^{-\pi i(\alpha+\frac{1}{k\tau})}\right)\left(e^{2\pi i (\alpha+\frac{1}{k\tau})};q_0\right)_{\infty}\left(q_0e^{-2\pi i (\alpha+\frac{1}{k\tau})};q_0\right)_{\infty}(q_0;q_0)_{\infty}\\
    &= -\frac{q^{-\frac{\alpha^2}{2}}e^{\pi i\alpha(1-\frac{2}{k})}}{\sqrt{-i\tau}}q^{\frac{1}{2k^2}-\frac{1}{2k}+\frac{1}{8}}_0\left(1 + O(q^{\frac{1}{k}}_0)  \right).
\end{align*}
Finally, the estimate for the $\eta$-function follows directly from the transformation law in Proposition \ref{transformprop}.\qed

We also need similar estimates for the Appell function near $\tau = 0$. Equation \eqref{mudef} gives
\begin{align*}
    \mu(5\tau,3\tau;12\tau) = \frac{q^{\frac{5}{2}}}{\vartheta(3\tau; 12\tau)}\sum_{m\in\mathbb{ Z}}\frac{(-1)^mq^{6m(m+1)}q^{3m}}{1-q^{12m}q^5}.
\end{align*}
Proposition \ref{transformprop}.3 implies that 
\begin{align}\label{muestimate}
    \mu(5\tau,3\tau;12\tau) = -q^2\frac{\mu\left(\frac{5}{12}, \frac{1}{4}; -\frac{1}{12\tau}\right)}{\sqrt{-12i\tau}} + \frac{h(2\tau;12\tau)}{2i}.
\end{align}
Before moving forward, we show that the integral $h(2\tau;12\tau)$ can be bounded by a standard Gaussian integral.

\begin{Lemma}\label{hlemma}
Let $0\leq \alpha <\frac{1}{2}$. Then as $\tau \to 0$ in a cone,
\begin{align*}
    h(\alpha\tau;\tau) \ll 1.
\end{align*}
\end{Lemma}
\proof
The proof  follows from the transformation law for $h$ given in Proposition \ref{transformprop}.4:
\begin{align*}
h(\alpha\tau;\tau) &= \frac{q^{\frac{\alpha^2}{2}}}{\sqrt{-i\tau}} \;h\left(\alpha;-\frac{1}{\tau}\right)= \frac{q^{\frac{\alpha^2}{2}}}{\sqrt{-i\tau}}\int^{\infty}_{-\infty}\frac{e^{-\frac{\pi i w^2}{\tau}}e^{-2\pi \alpha w}}{\textnormal{cosh}(\pi w)}dw\\
&=\frac{q^{\frac{\alpha^2}{2}}}{\sqrt{-i\tau}}\left(\int^{\infty}_{0}\frac{e^{-\frac{\pi i w^2}{\tau}}e^{-2\pi \alpha w}}{\textnormal{cosh}(\pi w)} - \int^{\infty}_{0}\frac{e^{-\frac{\pi i w^2}{\tau}}e^{2\pi \alpha w}}{\textnormal{cosh}(\pi w)}\right)dw,
\end{align*}
which implies that
\begin{align*}
\begin{split}\label{hupperbound}
|h(\alpha\tau;\;\tau)| &\leq \left\|\frac{q^{\frac{\alpha^2}{2}}}{\sqrt{-i\tau}}\right\|\; \left\|2\int^{\infty}_{0} e^{-\frac{\pi i w^2}{\tau}} \left(- \frac{e^{(2\alpha-1)\pi w}}{1+e^{-2\pi w}} + \frac{e^{-(2\alpha+1)\pi w}}{1+e^{-2\pi w}}  \right)dw\right\|.
\end{split}
\end{align*}
Since $0\leq \alpha<\frac{1}{2}$, the term in the parentheses is bounded above by a constant. Therefore,
\begin{align*}
    |h(\alpha\tau;\;\tau)|&\ll \left\|\frac{q^{\frac{\alpha^2}{2}}}{\sqrt{-i\tau}}\right\|\left\|\int^{\infty}_{0} e^{-\frac{\pi i w^2}{\tau}}dw\right\|\ll \frac{\sqrt{\tau}}{\sqrt{\tau}} = 1,
\end{align*}
Where we used the fact that $y>0$ and that $\int_{\mathbb{R}}e^{-\frac{yw^2}{|\tau|^2}}dw = \sqrt{\frac{\pi}{y}}|\tau|$. This leads to the claimed estimate as $\tau \to 0$.\qed

\subsection{Higher order estimates}\label{higherorderestsec}
\hspace{5mm} The main terms in the estimates of the previous section will not be sufficient in proving the growth of the $a(n)$, thus we need the following.

\begin{Lemma}\label{HIgherorderexpansion0lemma}
Let $\alpha$ be as in Lemma \ref{thetatestimateprelim}. Then as $\tau \to 0$ within a cone,
\begin{equation}\label{higherordfullpertheta}
    \vartheta(\alpha\tau; \tau) = -2i\frac{\textnormal{sin}(\pi \alpha)q^{\frac{1}{8}}_0}{\sqrt{-i\tau}}\left(1-a_1q_0 + a_3q_0^3 + O\left(q^4_0\right)\right),
\end{equation}

\begin{equation}\label{higherorderhalfpertheta}
    \vartheta\left(\frac{1}{2}+\alpha\tau; \tau\right) =-\frac{1}{\sqrt{-i\tau}}\left(1 -2\textnormal{cos}(2\pi\alpha)q^{\frac{1}{2}}_0+2\textnormal{cos}(4\pi\alpha)q^2_0+ O\left(q^4_0\right)\right),
\end{equation}
where,
\begin{align*}
    a_1&:= 1+2\textnormal{cos}(2\pi\alpha),\\
    a_3&:=  1+2\textnormal{cos}(2\pi \alpha) + 2\textnormal{cos}(4\pi \alpha).
\end{align*}
\end{Lemma}

\proof
Let $w:= e^{2\pi i \alpha}$. The proof of Eq. \eqref{higherordfullpertheta} follows directly by applying the technique in the proof of Lemma \ref{thetatestimateprelim} and observing that
\begin{align*}
    &(w;x)_{\infty} (xw^{-1};q_0)_{\infty} (x;x)_{\infty} \\
    &= (1-w)\left(1-(1+w+w^{-1})x+(1+w+w^{-1}+w^2+w^{-2})x^3 +O(x^4)\right).
\end{align*}
On the other hand, for Eq. \eqref{higherorderhalfpertheta}, we consider the associated Jacobi product
\begin{align*}
     &\left(wx^{-\frac{1}{2}};x\right)_{\infty} \left(w^{-1}x^{\frac{3}{2}};x\right)_{\infty} (x;x)_{\infty}\\
     &= -wx^{-\frac{1}{2}}+(1+w^2)-(w^{-1}+w^{3})x^{\frac{3}{2}}+O\left(x^3\right).
\end{align*}
Plugging this into the calculation in the proof of Lemma \ref{thetatestimateprelim} gives the result.\qed

\section{The $a(n)$}\label{Theanbaby}
 The function $T(\tau)$ defined in Example \ref{R3_3example} can be simplified greatly.
 \begin{Proposition}\label{bestrepofT}
 \begin{align}
 \begin{split}
 T(\tau) &= \frac{4iq^{\frac{417}{8}}}{\vartheta\left(\frac{1}{2}; \tau\right)}\frac{\vartheta^3(96\tau; 288\tau)}{\vartheta\left(\frac{1}{2}; 24\tau\right)}\\
 &\times\Bigg(q^{6}\frac{\vartheta(4\tau;16\tau)\vartheta(68\tau;96\tau)}{\vartheta\left(\frac{1}{2}+46\tau;96\tau\right)\vartheta\left(\frac{1}{2}+22\tau;96\tau\right)}-q^{-47}\frac{\vartheta(12\tau;16\tau)\vartheta(20\tau;96\tau)}{\vartheta\left(\frac{1}{2}+94\tau;96\tau\right)\vartheta\left(\frac{1}{2}+22\tau;96\tau\right)} \\
&+ q^{-39}\frac{\vartheta(4\tau;16\tau)\vartheta(20\tau;96\tau)}{\vartheta\left(\frac{1}{2}+46\tau;96\tau\right)\vartheta\left(\frac{1}{2}+70\tau;96\tau\right)}
-q^{-52}\frac{\vartheta(4\tau;16\tau)\vartheta(68\tau;96\tau)}{\vartheta\left(\frac{1}{2}+94\tau;96\tau\right)\vartheta\left(\frac{1}{2}+70\tau;96\tau\right)}\Bigg).
\end{split}
 \end{align}
 \end{Proposition}
 
\proof
Let
 \begin{align*}
     \overline{S}(\tau)&:= \sum^3_{r,s = 0}(-1)^rq^{Q(r,s)}\cdot \frac{\vartheta\left(4(s-r)\tau; 16\tau \right)\;\vartheta\Big((24(r+s)+44)\tau; 96\tau \Big)}{\vartheta\left(\frac{1}{2}+(24r +22)\tau ; 96\tau\right)\; \vartheta\left(\frac{1}{2}+(24s +22)\tau; 96\tau\right)} \\
     &:= \sum^3_{r,s=0}\varsigma(r,s;\tau). 
 \end{align*}
 It is clear that $\varsigma(r,r;\tau) = 0$ since $\vartheta(0;16\tau) = 0$. Furthermore, $Q(r,s) = Q(s,r)$. Using the fact that $\vartheta(-4(s-r)\tau;16\tau) = -\vartheta(4(s-r)\tau;16\tau)$, we deduce that
 \begin{align*}
 \varsigma(s,r;\tau)= -(-1)^{r+s}\varsigma(r,s;\tau).
\end{align*}
This tells us that we can write
\begin{align*}
    \overline{S}(q) = 2\left(\varsigma(1,0;\tau) + \varsigma(2,1;\tau) + \varsigma(3,0;\tau) + \varsigma(3,2;\tau) \right).
\end{align*}
We can then apply item $5$ of Prop. \ref{transformprop} to $\varsigma(2,1;\tau)$, $\varsigma(3,0;\tau)$ and $\varsigma(3,2;\tau)$ to complete the proof .\qed

Using this simplification in Proposition \ref{bestrepofT}, we will investigate the asymptotic growth of the coefficients of $R^3_3$ in the next section.

\subsection{The pole at $\tau=\frac{1}{2}$.}
As was claimed in Section \ref{higherorderestsec}, we require higher order asymptotic expansions to accurately determine the growth of the $a(n)$. We break the study  near $\tau=\frac{1}{2}$ into two parts: $T(\tau)$ and the Appell function, where Lemma \ref{HIgherorderexpansion0lemma} will prove useful for the study of $T(\tau)$.

\subsubsection{$T(\tau)$ near $\tau =\frac{1}{2}$.}
The first result involves the function $\vartheta(\frac{1}{2}; \tau)$, appearing in the denominator of $T(\tau)$.
First recall the eta multiplier, given by (See Theorem 5.8.1 of \cite{strombergbook})
\begin{align*}
    \varepsilon(A):= \begin{cases} \left(\frac{d}{|c|}\right)e^{\frac{\pi i}{12}\left((a+d-3)c-bd(c^2-1) \right)}& \text{if $c$ is odd}, \\ \\ \rho(c,d)\left(\frac{c}{|d|}\right)e^{\frac{\pi i}{12}\left((a-2d)c-bd(c^2-1) + 3d -3\right)}& \text{if $c$ is even},
    \end{cases}
\end{align*}
where $$A:= \begin{pmatrix} a&b\\c&d \end{pmatrix},$$  $\left(\frac{\bullet}{\bullet} \right)$ is the Jacobi symbol and 
$$ \rho(c,d) := \begin{cases}-1&\;\textnormal{if}\; c\leq 0, \; d<0,\\
1&\;\textnormal{else}.\end{cases} $$
\par Now we can prove the following.
\begin{Lemma}\label{halfperiosthetaneartauhalf}
 Define $w:= \tau - \frac{1}{2}$. As $w \to 0$ we have 
\begin{align*}
    \vartheta\left( \frac{1}{2}; \tau\right) =  2\frac{e^{-\frac{\pi i}{8}}e^{-\frac{\pi i}{16(\tau-\frac{1}{2})}}}{\sqrt{-2(\tau-\frac{1}{2})}}\left(1+O\left(e^{-\frac{\pi i}{2(\tau-\frac{1}{2})}}\right)\right).
\end{align*}
\end{Lemma}

\proof
Let $z:= -\frac{1}{4w}-\frac{1}{2}$ and define the matrices
\begin{align*}
    A&:= \begin{pmatrix} 1&0\\2&1\end{pmatrix},\;\;B:= \begin{pmatrix} 1&0\\1&1\end{pmatrix}.
\end{align*}

Then, since we have the well known formula $\vartheta\left(\frac{1}{2}; \tau\right) = 2\frac{\eta^2(2\tau)}{\eta(\tau)}$,
\begin{align*}
    \vartheta\left(\frac{1}{2}; Az\right) &= 2\frac{\eta^2(2Az)}{\eta(Az)} = 2\frac{\eta^2(B(2z))}{\eta(Az)}\\
    &= 2\frac{\varepsilon(B)^2(2z+1)^{\frac{1}{2}}\eta^2(2z)}{\varepsilon(A)\eta(z)} = 2\frac{\varepsilon(B)^2(2z+1)^{\frac{1}{2}}e^{\frac{8\pi i}{24}z} \left(e^{4\pi iz}; e^{4\pi i z}\right)^2_{\infty}}{\varepsilon(A)e^{\frac{2\pi iz}{24} }\left(e^{2\pi iz}; e^{2\pi i z}\right)_{\infty}}\\
    &=  2\frac{\varepsilon(B)^2(2z+1)^{\frac{1}{2}}e^{\frac{\pi i}{4}z} \left(e^{4\pi iz}; e^{4\pi i z}\right)^2_{\infty}}{\varepsilon(A)\left(e^{2\pi iz}; e^{2\pi i z}\right)_{\infty}}  \\
    &= 2\frac{e^{-\frac{\pi i}{8}}e^{-\frac{\pi i}{16w}}}{\sqrt{-2w}}\left(1+O\left(e^{-\frac{\pi i}{2w}}\right)\right),
\end{align*}
where $\epsilon(A) = \epsilon(B)^2= e^{-\frac{\pi i}{6}}$,
which proves the claim. \qed

\vspace{5mm}

\begin{Theorem}\label{mainterm}
Let $Q_0:= e^{-\frac{2\pi i}{\tau -\frac{1}{2}}}$ and write $\tau = u+iv$. Define $v:= \frac{1}{\sqrt{192n}}$ and let $M>0$ such that $\left|u-\frac{1}{2}\right|<Mv$. Then as $n\to \infty$
\begin{align*}
    T(\tau) = \frac{\sqrt{3}}{6\sqrt{\tau-\frac{1}{2}}}e^{\frac{\pi i}{4}}Q_0^{-\frac{1}{192}}\left(1+ O\left(e^{-\pi\sqrt{\frac{n}{12}}}\right)\right).
\end{align*}
\end{Theorem}
\proof
We focus our attention on
\begin{align*}
    \overline{S}(\tau)&:= q^{6}\frac{\vartheta(4\tau;16\tau)\vartheta(68\tau;96\tau)}{\vartheta\left(\frac{1}{2}+46\tau;96\tau\right)\vartheta\left(\frac{1}{2}+22\tau;96\tau\right)}-q^{-47}\frac{\vartheta(12\tau;16\tau)\vartheta(20\tau;96\tau)}{\vartheta\left(\frac{1}{2}+94\tau;96\tau\right)\vartheta\left(\frac{1}{2}+22\tau;96\tau\right)} \\
&+ q^{-39}\frac{\vartheta(4\tau;16\tau)\vartheta(20\tau;96\tau)}{\vartheta\left(\frac{1}{2}+46\tau;96\tau\right)\vartheta\left(\frac{1}{2}+70\tau;96\tau\right)}
-q^{-52}\frac{\vartheta(4\tau;16\tau)\vartheta(68\tau;96\tau)}{\vartheta\left(\frac{1}{2}+94\tau;96\tau\right)\vartheta\left(\frac{1}{2}+70\tau;96\tau\right)}.
\end{align*}
We refer to the first, second, third, and fourth terms as $S_1(\tau)$, $S_2(\tau)$, $S_3(\tau)$, and $S_4(\tau)$ respectively. That is, $\overline{S}(\tau) = S_1(\tau) + S_2(\tau) +S_3(\tau)+S_4(\tau)$. Notice that 
\begin{align*}
    S_1\left(\tau+\frac{1}{2}\right)+S_4\left(\tau+\frac{1}{2}\right) &= S_1(\tau)+S_4(\tau),\\
    S_2\left(\tau+\frac{1}{2}\right) + S_3\left(\tau +\frac{1}{2}\right) & = -S_1(\tau) -S_3(\tau). 
\end{align*}
Thus, we can capture the behavior near the cusp $\frac{1}{2}$ by investigating the behavior near $0$.
We can apply Lem. \ref{HIgherorderexpansion0lemma} to the $S_i(\tau)$, and we find that as $\tau \to 0$
\begin{align*}
 S_1(\tau) + S_4(\tau) = &-4\sqrt{6}\textnormal{sin}\left(\frac{\pi}{4}\right)\textnormal{sin}\left(\frac{17\pi}{24}\right)q^{\frac{7}{8\cdot96}}_0\Bigg((a_1+a_2-c_1-c_2)q_0^{\frac{1}{192}}\\
 & -\left(1+2\textnormal{cos}\left(\frac{17\pi }{12}\right)\right)(a_1+a_2-c_1-c_2)q^{\frac{3}{2\cdot{96}}}_0 + O\left(q^{\frac{5}{192}}_0\right)\Bigg),
\end{align*}
where,
\begin{align*}
    a_1&:= 2\textnormal{cos}\left(\frac{23\pi }{24}\right), \;\;\;a_2:= 2\textnormal{cos}\left(\frac{11\pi }{24}\right)\\
    c_1&:= 2\textnormal{cos}\left(\frac{47\pi }{24}\right),\;\;\;c_2:= 2\textnormal{cos}\left(\frac{35\pi }{24}\right).
\end{align*}
Similarly as $\tau \to 0$,
\begin{align*}
 S_2(\tau) + S_3(\tau) = &-4\sqrt{6}\textnormal{sin}\left(\frac{\pi}{4}\right)\textnormal{sin}\left(\frac{5\pi}{24}\right)q^{\frac{7}{8\cdot96}}_0\Bigg((h_1+h_2-l_1-l_2)q_0^{\frac{1}{192}}\\
 &-\left(1+2\textnormal{cos}\left(\frac{5\pi }{12}\right)\right)(h_1+h_2-l_1-l_2)q^{\frac{3}{2\cdot{96}}}_0 + O\left(q^{\frac{5}{192}}_0\right)\Bigg),
\end{align*}
where,
\begin{align*}
    h_1&:= 2\textnormal{cos}\left(\frac{23\pi }{24}\right), \;\;\;h_2:= 2\textnormal{cos}\left(\frac{35\pi }{24}\right)\\
    l_1&:= 2\textnormal{cos}\left(\frac{47\pi }{24}\right),\;\;\;l_2:= 2\textnormal{cos}\left(\frac{11\pi }{24}\right).
\end{align*}
We then observe that,
\begin{align*}
   \textnormal{sin}\left(\frac{\pi}{4}\right)\textnormal{sin}\left(\frac{17\pi}{24}\right)\Big(a_1+a_2-c_1-c_2\Big) -\textnormal{sin}\left(\frac{\pi}{4}\right)\textnormal{sin}\left(\frac{5\pi}{24}\right)\Big(h_1+h_2-l_1-l_2\Big) = 0, 
\end{align*}
and 
\begin{align*}
    &\textnormal{sin}\left(\frac{17\pi}{24}\right)\left(1+2\textnormal{cos}\left(\frac{17\pi }{12}\right)\right)\Big(a_1+a_2-c_1-c_2 \Big)\\
    &-\textnormal{sin}\left(\frac{5\pi}{24}\right)\left(1+2\textnormal{cos}\left(\frac{5\pi }{12}\right)\right)\Big(h_1+h_2-l_1-l_2 \Big)=  2\sqrt{2}.
\end{align*}

Thus, 
\begin{align}
\begin{split}\label{sbarest}
    \displaystyle{\lim_{q\to -1} \overline{S}(q)}&= \displaystyle{\lim_{q\to 1}\left(S_1(q) -S_2(q) -S_3(q)+S_4(q)\right)} \\&= \left(4\sqrt{6}\textnormal{sin}\left(\frac{\pi}{4}\right)\right)\left(2\sqrt{2}\right)Q^{\frac{19}{8\cdot96}}_0\left(1+O\left(Q^{\frac{1}{96}}_0\right)\right)\\&=8\sqrt{6}Q^{\frac{19}{8\cdot96}}_0\left(1+O\left(Q^{\frac{1}{96}}_0\right)\right).
\end{split}
\end{align}
We now turn our attention to the outside $\vartheta$-quotient on $T(\tau)$. Sending $\tau \to \tau + \frac{1}{2}$ and using the appropriate $\vartheta$ transformations, we have
\begin{align*}
\begin{split}
    \frac{4iq^{\frac{417}{8}}}{\vartheta\left(\frac{1}{2}; \tau \right)}\frac{\vartheta^3(96\tau; 288\tau)}{\vartheta\left(\frac{1}{2}; 24\tau\right)} &\to \;\;  -\frac{4iq^{\frac{417}{8}}e^{\frac{417\pi i}{8}}}{\vartheta\left(\frac{1}{2}; \tau +\frac{1}{2}\right)}\frac{\vartheta^3(96\tau; 288\tau)}{\vartheta\left(\frac{1}{2}; 24\tau\right)}.
    \end{split}
\end{align*}
Applying Lem. \ref{halfperiosthetaneartauhalf} and  Lem. \ref{thetatestimateprelim} leads to the near $\frac{1}{2}$ estimate
\begin{align}\label{forfactorest}
\begin{split}
 &-\frac{4iq^{\frac{417}{8}}e^{\frac{417\pi i}{8}}}{\vartheta\left(\frac{1}{2}; \tau +\frac{1}{2}\right)}\frac{\vartheta^3(96\tau; 288\tau)}{\vartheta\left(\frac{1}{2}; 24\tau\right)}\\ &\sim -4i\frac{\left(-2i\textnormal{sin}\left(\frac{\pi}{3}\right)\right)^3Q_0^{\frac{1}{8\cdot 96}}}{(-288i(\tau-\frac{1}{2}))^{\frac{3}{2}}}\;\frac{\sqrt{-2(\tau-\frac{1}{2})}Q_0^{-\frac{1}{32}}e^{\frac{\pi i}{4}}}{2}\;\left(-\sqrt{-24i\left(\tau-\frac{1}{2}\right)}\right)\\
   &= \frac{\sqrt{2}e^{\frac{\pi i}{4}}}{96\sqrt{\tau-\frac{1}{2}}}Q^{\frac{1}{8\cdot96}-\frac{1}{32}}_0.
 \end{split}
 \end{align}
Combining Eqs. \eqref{sbarest} and \eqref{forfactorest}, we obtain the full estimate for $T(\tau)$ near $\tau = \frac{1}{2}$:
\begin{align*}
    T\left(\tau \to \frac{1}{2}\right) &= \frac{\sqrt{3}}{6\sqrt{\tau-\frac{1}{2}}}e^{\frac{\pi i}{4}}Q_0^{\frac{20}{8\cdot96}-\frac{24}{8\cdot 96}}\left(1+ O\left(Q_0^{\frac{1}{96}}\right)\right)\\&=  \frac{\sqrt{3}}{6\sqrt{\tau-\frac{1}{2}}}e^{\frac{\pi i}{4}}Q_0^{-\frac{1}{192}}\left(1+ O\left(Q_0^{\frac{1}{96}}\right)\right),
\end{align*}
which proves the claim. \qed


\subsubsection{The Appell sum near $\tau = \frac{1}{2}$.}\label{appellsumnearminus1}
 In order to simplify our calculations, we introduce the notation $\dot{=}$ to mean \textit{equal up to a multiple of $O(1)$} and $c_n\;\dot{\sim}\;g_n$ to mean $c_n= M(g_n+o(1)),$ where $M=O(1)$. Since our Appell sum is invariant under the transformation $\tau \mapsto \tau + \frac{1}{2}$, it suffices to look at the behavior near $\tau = 0$. Thus, using the transformation law of Proposition \ref{transformprop}.3,
\begin{align}\label{R33nearzeroappellproof}
    \mu\left(2\tau+\frac{1}{2}, \frac{1}{2}; 24\tau   \right) = -\frac{q^{\frac{1}{12}}}{\sqrt{-24i\tau}}\mu\left(\frac{1}{12}+ \frac{1}{48\tau}, \frac{1}{48\tau}; -\frac{1}{24\tau} \right) + \frac{h(2\tau;24\tau)}{2i}.
\end{align}
 Looking solely at the remaining Appell sum gives
 \begin{align*}
    &\mu\left(\frac{1}{12}+ \frac{1}{48\tau}, \frac{1}{48\tau}; -\frac{1}{24\tau} \right) = \frac{e^{\frac{\pi i}{12}}e^{\frac{\pi i }{48\tau}}}{\vartheta\left(\frac{1}{48\tau}; -\frac{1}{24\tau} \right)}\sum_{n\in \mathbb{Z}}\frac{(-1)^nq^{\frac{n^2+n}{48}}_0q^{-\frac{n}{48}}_0}{1-e^{\frac{\pi i}{6}}q^{-\frac{1}{48}}_0q^{\frac{n}{24}}_0}\\
    &\dot{=}  \frac{q^{-\frac{1}{96}}_0}{q^{\frac{1}{8\cdot24}}_0q^{\frac{1}{96}}_0\left(1-q^{-\frac{1}{48}}_0 + O(q^{\frac{1}{48}}_0)\right) }\left(-\frac{e^{-\frac{\pi i}{6}}q^{\frac{1}{48}}_0}{1-e^{-\frac{\pi i}{6}}q_0^{\frac{1}{48}}}  + O(q^{\frac{1}{24}}_0)\right)\\
    &\dot{=} \frac{q^{-\frac{1}{96}}_0q_0^{\frac{1}{24}}}{q^{\frac{1}{8\cdot24}}_0q^{-\frac{1}{96}}_0\left(1 + O(q^{\frac{1}{48}}_0)\right) }\left(1+ O(q^{\frac{1}{48}}_0)\right) \dot{\sim}\;\; q_0^{\frac{7}{192}}.
 \end{align*}
 Plugging this back into Eq. \eqref{R33nearzeroappellproof} and using the estimate in Lemma \ref{hlemma}, we find as $\tau \to 0$,
 \begin{align*}\label{R33proofappellisanerrornear0}
     \mu\left(2\tau+\frac{1}{2}, \frac{1}{2}; 24\tau   \right) \ll 1.
 \end{align*}

Thus, we have the following.
\begin{Remark}
\normalfont The growth of $R^{(3)}_3$ near $\tau = \frac{1}{2}$ is determined by the estimate in  Theorem \ref{mainterm}.
\end{Remark}

\section{Growth on the minor arcs}
We now want to show that the growth at the other cusps is negligible to that  given in Theorem \ref{mainterm}. Thus, our target is to beat the bound exponentially
\begin{align*}
    \frac{Q_0^{-\frac{1}{96}}}{\sqrt{\tau -\frac{1}{2}}} \ll n^{\frac{1}{4}}e^{\pi\sqrt{\frac{n}{12}}},
\end{align*}
where we chose the parameterization $v = \frac{1}{\sqrt{192n}}$. Thus, we can incorporate the estimates away from $\frac{1}{2}$ into an error term, and ignore them in our final estimate for the $a(n)$.

\subsection{Bounding the Appell sum away from $\tau = \frac{1}{2}$}
\begin{Lemma}\label{logproductbound}
Let $|u|> Mv$, and let $a$ and $b$ be positive integers with $a<b$. Furthermore, let $v:= \frac{1}{\delta\sqrt{n}}$ with $\delta>0$ and for $M>0$ define the term $\varepsilon:= -\frac{1}{\sqrt{1+M^2}}+1>0$. Then as $n\to \infty$
\begin{equation}\label{firstguy}
 \frac{1}{\vartheta(a\tau;\; b\tau)},\; \frac{1}{\vartheta(\frac{1}{2}+a\tau;\; b\tau)}
\ll \frac{1}{n^{\frac{1}{4}}} e^{\frac{3\delta\sqrt{n}}{2\pi b}\Big(\frac{\pi^2}{6} - \varepsilon\Big)},
\end{equation}
and
\begin{equation}\label{thirdguy}
    \frac{1}{\vartheta(\frac{1}{2};b\tau)} \ll \frac{1}{n^{\frac{3}{4}}}\; e^{\frac{\delta\sqrt{n}}{\pi b}\left(\frac{\pi^2}{6} -\varepsilon \right)}.
\end{equation}
\end{Lemma}

\begin{Remark}\label{remarkforproofof5cusps}
    Notice that the right hand side of Eq. \eqref{firstguy} does not depend on $a$. Furthermore, the bounds above also hold for the functions $\vartheta(a\tau;b\tau)$, $\vartheta\left(\frac{1}{2};b\tau\right)$ and $\vartheta\left(\frac{1}{2}+a\tau;b\tau\right)$ which can by seen by replacing $\textnormal{Log}(\bullet)$ with $-\textnormal{
    Log}(\bullet)$ in the proof below. 
\end{Remark}

\proof
The proof uses the same ideas as in \cite{  BringDousseLoveMahl,DousseMertensAsymp,awayfromqrichard} to prove their bounds away from the dominant pole. We recall the Taylor expansion for $\textnormal{Log}(1 - z) = -\sum_{n\geq 1}\frac{z^n}{n}$. This implies,
\begin{align}\label{logequation1}
   &\textnormal{Log}\left(\frac{1}{(q^a;q^b)_{\infty}(q^{b-a};q^b)_{\infty}(q^b;q^b)_{\infty}}\right) =\sum_{n\geq 1}\frac{q^{an}+q^{(b-a)n}+ q^{bn}}{n(1-q^{bn})}.
\end{align}
The trick now, as described by many works such as \cite{BringDousseLoveMahl,DousseMertensAsymp,awayfromqrichard}, is to extract the first term in the sum, and add an extra term, which will be the first term in the expansion for
\begin{align*}
    \textnormal{Log}\left(\frac{1}{(|q|^a;|q|^b)_{\infty}(|q|^{b-a};|q|^b)_{\infty}(|q|^b;|q|^b)_{\infty}}\right).
\end{align*}
Explicitly, we have
\begin{align*}
    \sum_{n\geq 1}\frac{q^{an}+q^{(b-a)n}+ q^{bn}}{n(1-q^{bn})} = \sum_{n\geq 2}\frac{q^{an}+q^{(b-a)n}+ q^{bn}}{n(1-q^{bn})} + \frac{q^{a}+q^{b-a}+ q^b}{(1-q^{b})} \\
    + (|q|^{a}+|q|^{b-a}+ |q|^b)\left(\frac{1}{1-|q|^b}-\frac{1}{1-|q|^b}\right).
\end{align*}
Taking the absolute value of this equation and using the fact that $1-|q|^b\leq |1-q^b|$, we have the upper bound
\begin{align*}
    &\left|\textnormal{Log}\left(\frac{1}{(q^a;q^b)_{\infty}(q^{b-a};q^b)_{\infty}(q^b;q^b)_{\infty}}\right)\right| \leq \Bigg|\sum_{n\geq 2}\frac{q^{an}+q^{(b-a)n}+ q^{bn}}{n(1-q^{bn})} + \frac{q^{a}+q^{b-a}+ q^b}{(1-q^{b})}\\ &+ (|q|^{a}+|q|^{b-a}+ |q|^b)\left(\frac{1}{1-|q|^b}-\frac{1}{1-|q|^b}\right)\Bigg|\\
    &\leq \sum_{n\geq 2}\frac{|q|^{an}+|q|^{(b-a)n}+ |q|^{bn}}{n(1-|q|^{bn})} + \frac{|q|^{a}+|q|^{b-a}+ |q|^b}{|(1-q^{b})|} \\&\hspace{50mm}+ (|q|^{a}+|q|^{b-a}+ |q|^b)\left(\frac{1}{1-|q|^b}-\frac{1}{1-|q|^b}\right)\\
    & = \sum_{n\geq 1}\frac{|q|^{an}+|q|^{(b-a)n}+ |q|^{bn}}{n(1-|q|^{bn})} + (|q|^{a}+|q|^{b-a}+ |q|^b)\left(\frac{1}{|1-q^b|}-\frac{1}{1-|q|^b}\right)\\
    & = \textnormal{Log}\left(\frac{1}{(|q|^a;|q|^b)_{\infty}(|q|^{b-a};|q|^b)_{\infty}(|q|^b;|q|^b)_{\infty}} \right)\\
    &\hspace{50mm}+ (|q|^{a}+|q|^{b-a}+ |q|^b)\left(\frac{1}{|1-q^b|}-\frac{1}{1-|q|^b}\right) .
\end{align*}
The $\textnormal{Log}$ term can be estimated by the asymptotic formulas derived for the $\vartheta$-functions in Lemma \ref{thetatestimateprelim}. Namely, as $n\to \infty$
\begin{align}\label{abslogbound}
\begin{split}
    \textnormal{Log}\left(\frac{1}{(|q|^a;|q|^b)_{\infty}(|q|^{b-a};|q|^b)_{\infty}(|q|^b;|q|^b)_{\infty}}\right) &\ll \textnormal{Log}\left(\frac{\sqrt{\frac{b}{\delta\sqrt{n}}}}{2\textnormal{sin}(\pi\frac{a}{b})}\right) + \frac{\delta\pi\sqrt{n}}{4b}\\
    &\ll C_{a,b}+\textnormal{Log}\left(n^{-\frac{1}{4}}\right) + \frac{\delta\pi\sqrt{n}}{4b},
    \end{split}
\end{align}
where $C_{a,b}$ is a constant. Now we bound the fractions. Recall that we are away from the root of unity, $q = 1$ corresponding to $l=0$, by the amount $|u|>Mv$. Following the procedure on page 10 of \cite{BringDousseLoveMahl}, we can bound $\frac{1}{|1-q^b|}$ by using the fact that cosine is a decreasing function near 0. Namely, as $n\to\infty$
\begin{align*}
    |1-q^b|^2 &= 1 - 2\textnormal{cos}(2\pi b x)e^{-2\pi y b}+ e^{-4\pi by}
\gg 4\pi^2b^2y^2(1+M^2).
\end{align*}
This implies that
\begin{align}\label{fracfullabsvalest}
    \frac{1}{|1-q^b|}\ll \frac{1}{2\pi b y\sqrt{1+M^2}}.
\end{align}
For the other fraction, we have as $n\to \infty$
\begin{align}\label{partialfracabsvalest}
    \frac{1}{1-|q|^b}\sim \frac{1}{2\pi b y}.
\end{align}
Combining Eqs. \eqref{abslogbound}, \eqref{fracfullabsvalest}, and \eqref{partialfracabsvalest}, we have
\begin{align*}
    \textnormal{Log}\left(\frac{1}{(q^a;q^b)_{\infty}(q^{b-a};q^b)_{\infty}(q^b;q^b)_{\infty}}\right) \ll C_{a,b}+\log(n^{-\frac{1}{4}}) + \frac{\delta\pi\sqrt{n}}{4b} +\frac{3}{2\pi b y\sqrt{1+M^2}} -\frac{3}{2\pi b y},
\end{align*}
which proves the first part of Eq. \eqref{firstguy}.

The second part of Eq. \eqref{firstguy} follows by noticing that 
\begin{align*}
    &\textnormal{Log}\left(\frac{1}{(q^b;q^b)_{\infty}(-q^{b-a};q^b)_{\infty}(-q^a;q^b)_{\infty}}\right) = \sum_{n\geq 1} \frac{(-1)^{n}(q^a+q^{b-a})+q^b}{n(1-q^{bn})} \\
    &\ll \textnormal{Log}\left(\frac{1}{(|q|^a;|q|^b)_{\infty}(|q|^{b-a};|q|^b)_{\infty}(|q|^b;|q|^b)_{\infty}} \right)+ (|q|^{a}+|q|^{b-a}+ |q|^b)\left(\frac{1}{|1-q^b|}-\frac{1}{1-|q|^b}\right),
\end{align*}
as before.

For Eq. \eqref{thirdguy}, we have that
\begin{align*}\label{etaestimate2}
    \vartheta\left(\frac{1}{2},\tau\right) = -q^{\frac{1}{8}}(-1;q)_{\infty}(q^2;q^2)_{\infty} =-2q^{\frac{1}{8}}(-q;q)_{\infty}(q^2;q^2)_{\infty},
\end{align*}
which implies
\begin{align*}
    \left|\textnormal{Log}\left(\frac{1}{\vartheta\left(\frac{1}{2},\tau\right)}\right)\right| = \left|-\textnormal{Log}\left(-\frac{q^{-1/8}}{2}\right) +\sum_{n\geq 1}\frac{(-1)^nq^n}{n(1-q^n)} + \sum_{n\geq 1}\frac{q^{2n}}{n(1-q^{2n})}\right|\\
    \leq B + \sum_{n \geq 1}\frac{2|q|^n}{n(1-|q|^n)}\ll \textnormal{Log}(P^2(|q|),
\end{align*}
where $B$ is a constant and $P(q)=\frac{ q^{\frac{1}{24}}}{\eta(\tau)}$. 
Using Lemma $4.6$ of \cite{DousseMertensAsymp} we have
\begin{align*}
    P^2(|q|^b) \ll \frac{1}{n^{\frac{3}{4}}}\;e^{\frac{\delta\sqrt{n}}{\pi b}\left(\frac{\pi^2}{6}+ \frac{1}{\sqrt{1+M^2}}-1\right)}.
\end{align*}
This completes the proof of  Eq. \eqref{thirdguy}.\qed

We look at the non-normalized Appell sum
\begin{equation}\label{Adef}
    A\left(2\tau+\frac{1}{2}, \frac{1}{2}; 24\tau\right):= \vartheta\left(\frac{1}{2};24\tau\right)\mu\left(2\tau+\frac{1}{2}, \frac{1}{2}; 24\tau\right) = -q\sum_{n\in \mathbb{Z}}\frac{(-1)^nq^{12(n^2+n)}}{1+q^{24n+2}}.
\end{equation}
The following result shows that we can bound the above sum by a classical \textit{single} variable theta function, $\Theta(\tau)$. A similar result was also mentioned by the authors of \cite{BringDousseLoveMahl}, but was not carried out explicitly.
\begin{Proposition}\label{boundingmubytheta}
Let $\Theta(\tau):= \displaystyle{\sum_{n\in \mathbb{Z}}q^{n^2}}$. Then,
\begin{align*}
    \left|  A\left(2\tau+\frac{1}{2}, \frac{1}{2}; 24\tau\right)\right|\leq \frac{\Theta(iy)}{1-|q|^2}.
\end{align*}
\end{Proposition}

\proof
Splitting the sum in Eq. \eqref{Adef} into negative and positive index, and then recombining, we find
\begin{align*}
A\left(2\tau+\frac{1}{2}, \frac{1}{2}; 24\tau\right) = \frac{q}{1+q^2}+q\sum_{n>0}(-1)^nq^{12(n^2+n)}\left(\frac{1}{1+q^{24n+2}}+\frac{q^{-2}}{1+q^{24n-2}}\right).
\end{align*}
Since $|q|< 1$, we have that $1-|q| \leq |1+q|$. Combined with the fact that $|q|^m$ is a decreasing function in $m$, we have that,
\begin{align*}
    \left|  A\left(2\tau+\frac{1}{2}, \frac{1}{2}; 24\tau\right)\right|&\leq \frac{1}{1-|q|^2}+ \sum_{n>0}|q|^{12(n^2+n)}\left(\frac{2}{1+|q|^{24n-2}}\right)\\
    &\leq \frac{1}{1-|q|^2}\left( 1+ 2\sum_{n>0}|q|^{n^2}\right) = \frac{1}{1-|q|^2}\Theta(iy),
\end{align*}
which proves the claim.\qed

\begin{Remark}
Recall that $\Theta(\tau)$ is a holomorphic modular form for the group $\Gamma_0(4)$. $\Gamma_0(4)$ has three inequivalent cusps represented by $0, \frac{1}{2},\; \text{and} \; \infty$. 
\end{Remark}

Recall that in Section \ref{appellsumnearminus1} we computed the estimate near $0$ and $\frac{1}{2}$, which gives
\begin{align*}
    \mu\left(2\tau+\frac{1}{2}, \frac{1}{2}; 24\tau\right) \ll \frac{1}{\sqrt{|\tau|}}.
\end{align*}
Due to Proposition \ref{boundingmubytheta} and the corresponding remark, we can see we only need to check the growth of $\Theta(\tau)$ near $\infty$. Since $\Theta(\tau)$ is modular, it's growth at $\infty$ is at most $O(1)$. Since under the transformation $\tau \mapsto \tau +\frac{1}{2}$ the Jacobi theta remains unchanged, that is $\vartheta\left(\frac{1}{2};24\tau\right) \mapsto\vartheta\left(\frac{1}{2};24\tau\right)$, we can use Lemmas \ref{logproductbound} and \ref{boundingmubytheta} to obtain the following.

\begin{Theorem}\label{mufarfromminus1finalboundtheorem}
Let $M>0$ such that $0< yM < \left|x-\frac{1}{2}\right|$. Then there is a $\beta >0$ such that as $n\to \infty$,
\begin{align*}
    \mu\left(2\tau+\frac{1}{2}, \frac{1}{2}; 24\tau\right) \ll \frac{1}{\sqrt{|\tau|}}.
\end{align*}
\end{Theorem}

\par
\hspace{5mm} The estimates in Lemma \ref{logproductbound} alone are not sufficient to bound $T(\tau)$ away from the dominant pole since they do not provide accurate information about the decay of the $\vartheta$-functions near generic cusps $\frac{p}{h}$. However, we can use Lemma \ref{logproductbound}  in combination with a generalization of Lemma \ref{halfperiosthetaneartauhalf} to rule out contributions from cusps not equal to $\frac{1}{2}$.  Our goal is to prove the following.

\begin{Theorem}\label{boundforeverythingfarfromminus1}
 Let $M>0$ and let $v:= \frac{1}{\sqrt{192n}}$. For $Mv<\left|u-\frac{1}{2}\right|$, there exists a $\beta>0$ such that
\begin{align*}
    R^{(3)}_3(q) \ll  e^{\pi\sqrt{\frac{n}{12}}(1-\beta)},
\end{align*}
holds uniformly as $n\to \infty$.
\end{Theorem} 

We note additionally that $T(\tau)$ decays rapidly near 0.
\begin{Lemma}\label{decaynear0T}
As $\tau \to 0$ in a cone, 
\begin{align*}
     T(\tau) \ll |\tau|^{-\frac{1}{2}}q_0^{\frac{1}{96}}.
\end{align*}
\end{Lemma}

\proof
This follows easily from Lemma \ref{thetatestimateprelim}.\qed

\par
To deal with the other cusps, we recall the fact that $\vartheta(z;\tau)$ is a Jacobi form of weight and index $\frac{1}{2}$. Thus, the following properties hold \cite{Bringmannbook,jacobiformbook}:

\begin{Remark}[See Chapter 2, \cite{Bringmannbook}]\label{Jacobipropoftheta}
Let $A = \left(\begin{smallmatrix}a&b\\c&d \end{smallmatrix}\right)\in \textnormal{SL}_2(\mathbb{Z})$ and $\lambda, k \in \mathbb{Z}$. Then
\begin{align}
\begin{split}\label{jacobitranstheta}
    \vartheta\left(\frac{z}{c\tau+d};\frac{a\tau+b}{c\tau+d}\right) &= \chi(A)(c\tau+d)^{\frac{1}{2}}e^{\frac{\pi i cz^2}{c\tau+d}}\vartheta\left(z;\tau\right),
\end{split}
    \\
\begin{split}\label{elliptictranstheta}
    \vartheta\left(z+\lambda\tau+k;\tau \right) &= \beta(\lambda) e^{-\pi i\left(\lambda^2\tau +2\lambda z\right) } \vartheta\left(z;\tau\right),
\end{split}
\end{align}
where $\chi$ and $\beta$ are multipliers.
\end{Remark}
\par

\subsection{Growth formulas for theta function}

\hspace{5mm} We write down the growth of the Jacobi theta functions up to a constant near cusps $\frac{p}{h}$.  Let $\frac{p}{h}$ be a cusp in reduced form, and denote the set of divisors of a natural number $N$ by $\textnormal{div}(N)$.
\begin{Lemma}\label{a1b1}
Let $\tau := u+iv$ and suppose that $\left|u-\frac{p}{h}\right|$ is uniformly bounded by $v$. Let $a_1$ and $b_1$ be non-zero natural numbers with $a_1<b_1$. Then as $v\to 0$,
\begin{align*}
    \vartheta(a_1\tau;b_1\tau)&\;\dot{\sim}\;= v^{-\frac{1}{2}}e^{-\frac{\pi i}{b_1\Tilde{H}^2w}\left(\{\frac{\Tilde{H}\Tilde{\gamma}}{\Tilde{h}}\}^2 +\frac{1}{4}-\{\frac{\Tilde{H}\Tilde{\gamma}}{\Tilde{h}}\}\right)},
\end{align*}
where $$\Tilde{\gamma}:= \frac{a_1p\pmod h}{\gcd(h,a_1p\pmod h)}, \hspace{5mm} \Tilde{\Gamma}:= \frac{b_1p\pmod h}{\gcd(h,b_1p\pmod h)},$$ $$ \Tilde{h}:= \frac{h}{\gcd(h,a_1p\pmod h)}, \hspace{5mm} \Tilde{H}:= \frac{h}{\gcd(h,b_1p\pmod h)}.$$

\end{Lemma}

\begin{Remark}
\normalfont Notice that if $h=0$, that is, we approach the origin, then 
$$ \Tilde{\gamma} = 0, \;\Tilde{h}=1,\;\textnormal{and}\; \Tilde{H}=1.$$ As a result, we recover the same exponential term $q_0^{\frac{1}{8b_1}}$ as in Prop. \ref{thetatestimateprelim}.
\end{Remark}

Similarly, we have to address $\vartheta$-functions with a shifted elliptic variable.
\begin{Lemma}\label{a1b1half}
Let $\tau := u+iv$ and suppose that $\left|u-\frac{p}{h}\right|$ is uniformly bounded by $v$. Let $a_1$ and $b_1$ be non-zero natural numbers with $a_1<b_1$. Then as $v\to 0$,
\begin{align*}
    \vartheta\left(\frac{1}{2}+a_1\tau;b_1\tau\right)&\;\dot{\sim}\; v^{-\frac{1}{2}}e^{-\frac{\pi i}{b_1\Tilde{H}^2w}\left(\left\{\frac{\Tilde{H} \phi}{\omega}\right\}^2 +\frac{1}{4}-\left\{\frac{\Tilde{H}\phi}{\omega}\right\}\right)},
\end{align*}
where $$\phi:= \frac{h +2a_1p\pmod{2h}}{\gcd(2h,h +2a_1p\pmod{2h})} \hspace{5mm}\textnormal{and}\hspace{5mm}\omega:= \frac{2h}{\gcd(2h,h +2a_1p\pmod{2h})}.$$ 
\end{Lemma}
\begin{Remark}
\normalfont Note that if $p=0$, 
$$\phi= 1,\; \omega = 2,\;\textnormal{and}\; \Tilde{H}=1.$$ Therefore,
the main exponential term cancels and we are only left with the polynomial growth which we previously saw in Prop. \ref{thetatestimateprelim}.
\end{Remark}

\begin{proof}[Proof of Thm. \ref{a1b1}]
Let $$A:= \begin{pmatrix}\Tilde{\Gamma}&b\\\Tilde{H}&d \end{pmatrix},$$ where $b,d$ are chosen such that $\det(A) = 1$ which is allowed since $\Tilde{\Gamma}$ and $\Tilde{H}$ are relatively prime. Let $w:= \tau+\frac{p}{h}.$ Define $$\sigma:= -\frac{1}{b_1\Tilde{H}^2w}-\frac{d}{\Tilde{H}}.$$ Notice that $$\lim_{v\to 0}A(\sigma) = \lim_{v\to 0}\frac{\Tilde{\Gamma}}{\Tilde{H}}+b_1w =\frac{\Tilde{\Gamma}}{\Tilde{H}}. $$
Notice that applying the Jacobi transformations implies
$$\vartheta(a_1\tau;b_1\tau)\;\dot{=}\;\vartheta\left(a_1w+\frac{\Tilde{\gamma}}{\Tilde{h}} ;b_1w+\frac{\Tilde{\Gamma}}{\Tilde{h}}\right). $$
We want to study this function as $v\to 0$. We can write
\begin{align}\label{eq1}
    \vartheta\left(a_1w+\frac{\Tilde{\gamma}}{\Tilde{h}} ;b_1w+\frac{\Tilde{\Gamma}}{\Tilde{h}}\right) = \vartheta\left(\frac{z}{\Tilde{H}\sigma +d};A(\sigma)\right),
\end{align}
where $$ z:= (\Tilde{H}\sigma+d)\left(a_1w+\frac{\Tilde{\gamma}}{\Tilde{h}}\right) = -\frac{\Tilde{\gamma}}{\Tilde{h}\Tilde{H}b_1w}+ O(1).$$
Using the first Jacobi transform on Eq. \eqref{eq1} yields
\begin{align*}
    \vartheta\left(\frac{z}{\Tilde{H}\sigma +d};A(\sigma)\right) \;&\dot{=}\; v^{-\frac{1}{2}}e^{\frac{\pi i \Tilde{H}z^2}{\Tilde{H}\sigma+d}}\vartheta\left(z;\sigma\right)\\
    &=v^{-\frac{1}{2}}e^{\frac{\pi i \Tilde{H}z^2}{\Tilde{H}\sigma+d}}\vartheta\left(a_1\Tilde{H}\sigma w  +\frac{\Tilde{H}\Tilde{\gamma}}{\Tilde{h}}\sigma +a_1dw +\frac{d\Tilde{\gamma}}{\Tilde{h}}; \sigma\right)\\
    &\dot{=}\; v^{-\frac{1}{2}}e^{\frac{\pi i \Tilde{H}z^2}{\Tilde{H}\sigma+d}}e^{-\pi i \left(\floor{\frac{\Tilde{H}\Tilde{\gamma}}{\Tilde{h}}}^2\sigma +2\floor{\frac{\Tilde{H}\Tilde{\gamma}}{\Tilde{h}}}\left(\Lambda +\left\{\frac{\Tilde{H}\Tilde{\gamma}}{\Tilde{h}}\right\}\sigma\right)\right)} \vartheta\left(\left\{\frac{d\Tilde{\gamma}}{\Tilde{h}}\right\}+ \left\{\frac{\Tilde{H}\Tilde{\gamma}}{\Tilde{h}}\right\}\sigma+\Lambda; \sigma\right),
\end{align*}
where in the last step we used the elliptic property (the second Jacobi transformation) and we defined $\Lambda:= a_1w\left(\Tilde{H}\sigma+d\right).$ Notice that $\Lambda = O(1).$ We now apply the Jacobi triple product to find that
\begin{align*}
    \vartheta\left(\left\{\frac{d\Tilde{\gamma}}{\Tilde{h}}\right\}+ \left\{\frac{\Tilde{H}\Tilde{\gamma}}{\Tilde{h}}\right\}\sigma+\Lambda; \sigma\right) \;&\dot{=}\;e^{\frac{\pi i }{4}\sigma}e^{-\pi i(\{\frac{d\Tilde{\gamma}}{\Tilde{h}}\}+ \{\frac{\Tilde{H}\Tilde{\gamma}}{\Tilde{h}}\}\sigma+\Lambda)}\\
    &\hspace{15mm} \times \left(e^{2\pi i \sigma};e^{2\pi i \sigma}\right)_\infty \left(e^{2\pi i(\{\frac{d\Tilde{\gamma}}{\Tilde{h}}\}+ \{\frac{\Tilde{H}\Tilde{\gamma}}{\Tilde{h}}\}\sigma+\Lambda)};e^{2\pi i \sigma}\right)_\infty\\
    &\hspace{15mm} \times \left(e^{2\pi i \sigma}e^{-2\pi i(\{\frac{d\Tilde{\gamma}}{\Tilde{h}}\}+ \{\frac{\Tilde{H}\Tilde{\gamma}}{\Tilde{h}}\}\sigma+\Lambda)};e^{2\pi i \sigma}\right)_\infty\\
    &\dot{\sim}\; e^{\frac{\pi i }{4}\sigma}e^{-\pi i(\{\frac{d\Tilde{\gamma}}{\Tilde{h}}\}+ \{\frac{\Tilde{H}\Tilde{\gamma}}{\Tilde{h}}\}\sigma+\Lambda)} \;\dot{\sim} \; e^{-\frac{\pi i}{b_1\Tilde{H}^2w}\left(\frac{1}{4}-\{\frac{\Tilde{H}\Tilde{\gamma}}{\Tilde{h}}\}\right)}.
\end{align*}
Furthermore, noting that $\sigma w =O(1),$
\begin{align*}
    e^{\frac{\pi i \Tilde{H}z^2}{\Tilde{H}\sigma+d}}e^{-\pi i \left(\floor{\frac{\Tilde{H}\Tilde{\gamma}}{\Tilde{h}}}^2\sigma +2\floor{\frac{\Tilde{H}\Tilde{\gamma}}{\Tilde{h}}}\left(\Lambda+\{\frac{\Tilde{H}\Tilde{\gamma}}{\Tilde{h}}\}\sigma\right)\right)} \;\dot{\sim}\; e^{-\frac{\pi i \Tilde{\gamma}^2}{\Tilde{h}^2b_1w}}e^{\frac{\pi i}{b_1\Tilde{H}^2w}\left(\floor{\frac{\Tilde{H}\Tilde{\gamma}}{\Tilde{h}}}^2   +2\floor{\frac{\Tilde{H}\Tilde{\gamma}}{\Tilde{h}}}\{\frac{\Tilde{H}\Tilde{\gamma}}{\Tilde{h}}\} \right)}.
\end{align*}
Combining the last two equations gives the desired result.\end{proof}
\begin{proof}[Proof of Thm. \ref{a1b1half}]
We begin again by writing $w:= \tau +\frac{h}{k}.$ Then, by Prop. \ref{thetatestimateprelim}
$$ \vartheta\left(\frac{1}{2}+a_1\tau;b_1\tau\right) \dot{=} \;\vartheta\left(\frac{\phi}{\omega}+a_1\tau;\frac{\Tilde{\Gamma}}{\Tilde{H}}+b_1\tau\right).$$The proof is identical to the proof of Thm. \ref{a1b1} with $\Tilde{\gamma}\mapsto \phi$ and $\Tilde{h}\mapsto \omega$. \end{proof}
We now need to find similar formulas for theta functions whose elliptic variable is constant. We prove the following.
\begin{Lemma}\label{genericgrowthhalfperiod}
Let $\tau:= u+iv$. Suppose $\left|u-\frac{p}{h}\right|$ is uniformly bounded by $v$ for all $n$. Then we have  as $n\to \infty$
\begin{align*}
    \vartheta\left(\frac{1}{2}; \tau\right) \dot{\sim} \ \begin{cases} v^{-\frac{1}{2}}e^{-\frac{\pi }{4h^2v}}&\textnormal{if} \; h \;\textnormal{even}, \\
     v^{-\frac{1}{2}}&\textnormal{if} \; h \;\textnormal{odd} .\end{cases}
\end{align*}
\end{Lemma}
\proof
Let $A:= \begin{pmatrix}p&b\\h&d \end{pmatrix}$ with $b,d$ so that $A\in \textnormal{SL}_2(\mathbb{Z})$, which exist since $(p,h) = 1$. Define $z:= \frac{h\sigma+d}{2}$ with $\sigma:= -\frac{1}{h^2w}-\frac{d}{h}$ for $ w \in \mathbb{H}$. Notice that 
\begin{align*}
    A(\sigma) = \frac{p}{h}+w(pd-hb) = \frac{p}{h}+w.
\end{align*}
Thus, $\displaystyle{\lim_{w\to 0}}A(\sigma) = \frac{p}{h}$. Regardless of whether $h$ is even or odd, using Eq. \eqref{jacobitranstheta} yields 
\begin{align}\label{thetahalfusecases}
    \vartheta\left(\frac{1}{2}; A(\sigma)\right)& \dot{=}\;\; (h\sigma+d)^{\frac{1}{2}}e^{\pi i\frac{h^2\sigma+hd}{4}} \vartheta\left(\frac{h\sigma+d}{2};\sigma \right).
\end{align}
If $h$ is even, $d$ must be odd. Therefore, Eq. \eqref{elliptictranstheta} implies that Eq. \eqref{thetahalfusecases} reduces to,
\begin{align*}
   \vartheta\left(\frac{1}{2}; A(\sigma)\right)&\dot{=}\;\; \left(-\frac{1}{hw}\right)^{\frac{1}{2}}\vartheta\left(\frac{1}{2};\sigma \right),
\end{align*}
which upon using the Jacobi triple product and taking the limit $w \to 0$ gives
\begin{align*}
      \vartheta\left(\frac{1}{2}; A(\sigma)\right) \dot{\sim}  \left(-\frac{1}{hw}\right)^{\frac{1}{2}}e^{-\frac{\pi i}{4h^2w}}.
\end{align*}
Subbing in $w = \tau - \frac{p}{h}$ proves the first claim. \par The second case when $h$ is odd has two separate situations to contend with, depending on whether $d$ is odd or even. Assume first that $d$ is odd. Then $\frac{h}{2}$ and $\frac{d}{2}$ are both half integers. Therefore using Remark \ref{Jacobipropoftheta} as before, we have
\begin{align*}
\begin{split}\label{doddhalfperiodgenericcusp}
     \vartheta\left(\frac{1}{2}; A(\sigma)\right)& \dot{=}\;\; (h\sigma+d)^{\frac{1}{2}}e^{\pi i\frac{h^2\sigma}{4}}\vartheta\left(\left(\floor[\bigg]{\frac{h}{2}}+\frac{1}{2}\right)\sigma + \floor[\bigg]{\frac{d}{2}} +\frac{1}{2}; \sigma\right)\\
     &\dot{=}\;\; (h\sigma+d)^{\frac{1}{2}}e^{-\pi i\sigma\left(\floor{\frac{h}{2}}^2+\floor{\frac{h}{2}}-\frac{h^2}{4} \right)}\vartheta\left(\frac{\sigma+1}{2}; \sigma\right),
\end{split}
\end{align*}
where $\floor{\bullet}$ is the floor function. Using the Jacobi product again and taking the limit as $w \to 0$, we have 
\begin{align}
    \vartheta\left(\frac{1}{2}; A(\sigma)\right)&\dot{\sim}  (-hw)^{-\frac{1}{2}}e^{\frac{\pi i}{h^2w}\left(-\frac{h^2}{4}+\floor{\frac{h}{2}}^2+\floor{\frac{h}{2}}+\frac{1}{4}\right)} = (-2w)^{-\frac{1}{2}}.
\end{align}
The last step follows since $h$ odd implies $\floor{\frac{h}{2}} = \frac{h-1}{2}$. If $d$ is even, the only thing that changes in Eq. \eqref{doddhalfperiodgenericcusp} is that  $\vartheta\left(\frac{\sigma + 1}{2}; \sigma \right)$  becomes $\vartheta\left(\frac{\sigma  }{2}; \sigma\right)$, which both yield the same estimate up to a constant factor as $w\to 0$ by examining the triple product representations.\qed

With the previous lemma, we have the following useful corollary.
\begin{Corollary}\label{24halfnearph}
Let $\tau:= u+iv$. Suppose $|u|$ is uniformly bounded by $v$ for all $n$. Then we have the lower bound as $n\to \infty$
$$ \vartheta\left(\frac{1}{2}; 24\tau\right)\gg  v^{-\frac{1}{2}}e^{-\frac{\pi}{96\Tilde{h}^2v}},$$
where $\Tilde{h}:= \frac{h}{g}$ where $g:= \gcd(24p\pmod{h},h)$.
\end{Corollary}
\begin{proof}
If $h|24$, Prop. \ref{thetatestimateprelim} implies that $ \vartheta\left(\frac{1}{2}; 24\tau\right)$ is asymptotic to a polynomial in $\tau$, and thus the claim follows. If $h\not| 24$, defining $\tau := w+\frac{p}{h}$ and applying the transformations in Remark \ref{jacobitranstheta} gives that 
$$\vartheta\left(\frac{1}{2};24\tau\right)\dot{=}\vartheta\left(\frac{1}{2};24w+\frac{\Tilde{\gamma}}{\Tilde{h}}\right), $$
where $\Tilde{\gamma}:= \frac{24p \pmod{h}}{g}$.  The machinery of the the proof of Lemma \ref{genericgrowthhalfperiod} with $p:= \Tilde{\gamma}$ and $w:= 24w$ now applies. If $\Tilde{h}$ is odd, we are done. If $\Tilde{h}$ is even, we find that as $n\to \infty$ in a cone
$$ \vartheta\left(\frac{1}{2}; 24\tau\right) \dot{\sim}  \left(-24\tilde{h}w\right)^{-\frac{1}{2}}e^{-\frac{\pi i}{96 \Tilde{h}^2w}}\gg v^{-\frac{1}{2}}e^{-\frac{\pi i}{96\Tilde{h}^2v}},$$
where the last step follows from the fact that $\mathrm{Im}(w) = \mathrm{Im}(\tau) = v.$\end{proof}

\subsection{Proof of Lemma \ref{boundforeverythingfarfromminus1}}
\hspace{5mm}As stated at the beginning of this section, we claimed that there are only a finite number of cusps we need to check. The following proposition gives us a rough bound for this number, but it more importantly tells us that all of the cusps that could cause a large pole have $h$ even. As a reminder, we have already checked explicitly that $T(\tau)$ decays exponentially near $0$ in Lemma \ref{decaynear0T}, which means we do not have to investigate this case in the proof below.

\begin{Proposition}\label{roughcuspbound}
Let $\delta := \sqrt{192}$. With this choice, the only cusps that could lead to $T(\tau)$ having larger growth than that at $\frac{1}{2}$ are cusps $\frac{p}{h}$ with $h\leq24$ even. 
\end{Proposition}

\proof
We use the notation from the proof of Theorem \ref{mainterm}, where we saw that we could write $T(\tau)$ as
\begin{align*}
    T(\tau) \dot{=} \frac{\vartheta^3(96\tau;288\tau)}{\vartheta\left(\frac{1}{2};\tau\right)\vartheta\left(\frac{1}{2};24\tau\right)} \overline{S}(\tau).
\end{align*}
Referring back to Remark \ref{remarkforproofof5cusps} we can use Lemma \ref{logproductbound} to bound the combination
\begin{align*}
 D(\tau):= \vartheta^3(96\tau;288\tau)\overline{S}(\tau).
\end{align*}
Recalling that each one of the four terms in $\overline{S}(\tau)$ is of the form
\begin{align*}
    \frac{\vartheta(\bullet;16\tau)\vartheta(\bullet; 96\tau)}{\vartheta\left(\frac{1}{2}+\bullet;96\tau\right)\vartheta\left(\frac{1}{2}+\bullet; 96\tau\right)},
\end{align*}
and that the bounds in Lemma \ref{logproductbound} only depend on the factor in the second slot, we find that

\begin{align*}
    D(\tau)&\ll n^{-\frac{3}{2}}e^{\sqrt{192}\sqrt{n}\left(\frac{9}{2\cdot 288\pi} + \frac{3}{2\cdot16\pi}+ \frac{9}{2\cdot96\pi}\right)\left(\frac{\pi^2}{6}-\varepsilon\right)}
    = n^{-\frac{3}{2}}e^{\sqrt{192}\sqrt{n}\left(\frac{5}{32\pi}\right)\left(\frac{\pi^2}{6}-\varepsilon\right)}\\
    &= n^{-\frac{3}{2}}e^{\frac{5\sqrt{192}\pi}{192}\sqrt{n}\left(1-\frac{6}{\pi^2}\varepsilon\right)} = n^{-\frac{3}{2}}e^{\frac{5\pi}{4}\sqrt{\frac{n}{12}}\left(1-\frac{6}{\pi^2}\varepsilon\right)} ,
\end{align*}
where $\varepsilon = 1-\frac{1}{\sqrt{1+M^2}}$. We define 
\begin{align*}
    G(\tau):= \frac{1}{\vartheta\left(\frac{1}{2}; \tau\right)\vartheta\left(\frac{1}{2}; 24\tau\right)}.
\end{align*}
If $h$ is odd, by Lemma \ref{genericgrowthhalfperiod} and  Cor. \ref{24halfnearph}, we have that
$$ G(\tau)\ll ve^{\frac{\pi}{2\Tilde{h}^2}\sqrt{\frac{n}{12}}}.$$ Using Lemma \ref{logproductbound} on $D$ gives
$$ G(\tau)D(\tau) \ll v^{-\frac{1}{2}}e^{\pi\sqrt{\frac{n}{12}}\left(\frac{1}{2\Tilde{h}^2}+\frac{5}{4}-\frac{15\varepsilon}{2\pi^2}\right)}.$$ We want to $$ \frac{1}{2\Tilde{h}^2}+\frac{5}{4}-\frac{15\varepsilon}{2\pi^2}<1.$$ If we chose $\varepsilon = 0.99$, then (for example) all $\Tilde{h}>0.998$ satisfy the above inequality. That is, the inequality is always true since $\Tilde{h}\geq 1.$
\par Thus, we only need to consider even $h$. Applying Lemma \ref{genericgrowthhalfperiod} in the even case and Prop. \ref{24halfnearph} gives

\begin{align*}
    G(\tau)\ll v e^{\pi\sqrt{192n}\left(\frac{1}{96\Tilde{h}^2}+\frac{1}{4h^2}\right)} = ve^{\pi\sqrt{192n}\left(\frac{h^2+24\Tilde{h}^2}{96(\Tilde{h}h)^2}\right)} \ll ve^{\pi\sqrt{\frac{n}{12}}\left(\frac{12\cdot25}{h^2}\right)},
\end{align*}
where in the last line we used that $h\leq 24\Tilde{h}.$ Combining this with the previous estimate for $D$, we find
\begin{align*}
   G(\tau)D(\tau)\ll  e^{\pi\sqrt{\frac{n}{12}}\left(\frac{5}{4}-\frac{15}{2\pi^2}\varepsilon+\frac{12\cdot25}{h^2}\right)}.
\end{align*}
We want that
$$ \frac{5}{4}-\frac{15}{2\pi^2}\varepsilon+\frac{12\cdot 25}{h^2}< 1, $$ which is satisfied for 
$$ h> 5\sqrt{\frac{24}{\frac{15}{\pi^2}\varepsilon-\frac{1}{2}}}.$$ Choosing $\varepsilon$ again $\varepsilon = 0.99$, we find that (for example) all $ h>24.5$ satisfy the inequality. Thus, only poles with $h\leq 24$ and $h$ even can cause growth as large as the major arc case $h=2$.\qed

Let $\textbf{a}:= (a_1,a_2,a_3,a_4)\in \mathbb{N}^4$ and let us define the set of reduced fractions
$$\mathscr{X}(j):= \left\{\frac{p}{h} \Big| (p,h) = 1 \; \textnormal{and}\; h\; \textnormal{even} \; \textnormal{with}\; h\leq j\right\}.$$ For example,
$$\mathscr{X}(4) = \mathscr{X}(5)= \left\{\frac{1}{4},\frac{1}{2},\frac{3}{4}\right\}.$$ Define the functions $P_j: \{0,1,2,3,4\}\times\mathbb{Q}\to \mathbb{Q}$ by \begin{equation*} P_j\left(\frac{h}{p}\right):= \frac{\{\beta_j\}^2+\frac{1}{4}-\{\beta_j\}}{H_j^2},\end{equation*} where
\begin{align*}
    H_j:= \begin{cases}\frac{h}{\gcd(h,288p\pmod{h})}&\; \textnormal{if}\; j=0,\\ \frac{h}{\gcd(h,16p\pmod{h})}&\; \textnormal{if}\; j=1,\\ \frac{h}{\gcd(h,96p\pmod{h})}&\; \textnormal{if}\; j=2,3,4,\end{cases}\;
\textnormal{and}\;\;
 \beta_j:= \begin{cases}\frac{96p\pmod{h}}{\gcd(h,288p\pmod{h})}&\; \textnormal{if}\; j=0,\\ \frac{a_1p\pmod{h}}{\gcd(h,16p\pmod{h})}&\; \textnormal{if}\; j=1,\\ \frac{a_2p\pmod{h}}{\gcd(h,96p\pmod{h})}&\; \textnormal{if}\; j=2,\\ \frac{h+2a_ip\pmod{2h}}{2\gcd(h,96p\pmod{h})}&\; \textnormal{if}\; j=3,4.\end{cases}
\end{align*}
We finally define the function $F:\mathbb{N}^4\times\mathbb{Q}\to \mathbb{Q}$ by
\begin{align*}
 F\left(\textbf{a},\frac{h}{p}\right)&:= \frac{24}{h^2}+\frac{\gcd(h,24p\pmod{h})^2}{h^2}-P_0\left(\frac{h}{p}\right)-6P_1\left(\frac{h}{p}\right)\\ &\hspace{15mm}-P_2\left(\frac{h}{p}\right)+P_3\left(\frac{h}{p}\right)+P_4\left(\frac{h}{p}\right).
\end{align*}
We are interested in the values $F\left(\textbf{a},\frac{h}{p}\right)$, where $\frac{h}{p}\in \mathscr{X}(24).$ As such, we have the following lemma.
\begin{Lemma}\label{Fboundedby2}
If $\textbf{a}= (4,20,46,70),\;(4,68,46,22),\;(4,68,94,70),\; \textnormal{or}\; (12,20,94,22) $, then for all $\frac{h}{p}\in \mathscr{X}(24)$, $$ F\left(\textbf{a},\frac{h}{p}\right)<2. $$
\end{Lemma}
\begin{proof}
ince the set $\mathscr{X}(24)$ is finite, the proof amounts to checking all of the possible cases for $h$ and $k$. There are $62$ cases to check for each of the vectors $\textbf{a}$. Since the evaluation of the function $F$ is modular arithmetic, this can be checked with a computer with a simple for-loop procedure, for example in MAPLE. We find that the larges value across these \textbf{a} is given by $\frac{13}{9}$.\end{proof}


We are now in a position to prove our main result for the minor arcs.

\begin{proof}[Proof of Lemma \ref{boundforeverythingfarfromminus1}]
Based on the form $T(\tau)$ given in Eq. \eqref{bestrepofT}, we need to investigate functions of the form
\begin{align*}
    D(\textbf{a};\tau):= \frac{\vartheta^3(96\tau;288\tau)\;\vartheta(a_1\tau;16\tau)\;\vartheta(a_2\tau; 96\tau)}{\vartheta\left(\frac{1}{2}+a_3\tau;96\tau\right)\;\vartheta\left(\frac{1}{2}+a_4\tau; 96\tau\right)},
\end{align*}
for even integers $a_i$. According to Prop. \ref{roughcuspbound} we need to check the growth of $T(\tau)$ near cusps $\frac{p}{h}\in \mathscr{X}(24)$. We know from Lemma \ref{genericgrowthhalfperiod} and Cor. \ref{24halfnearph} with $v=\frac{1}{\sqrt{192n}}$ that
\begin{align}\label{halfyboundy}
    G(\tau)=\frac{1}{\vartheta\left(\frac{1}{2}; \tau\right)\; \vartheta\left(\frac{1}{2}; 24\tau\right)} \ll v e^{\frac{\pi}{48}\sqrt{192n}\left(\frac{12}{h^2}+\frac{1}{2\Tilde{h}^2}\right)}\ll ve^{\pi\sqrt{\frac{n}{12}}\left(\frac{24\Tilde{h}^2+h^2}{2(h\Tilde{h})^2}\right)},
\end{align}
where $$\Tilde{h}:= \frac{h}{\gcd(h,\; 24p\pmod{h})}.$$
On the other hand, directly applying Lemmas \ref{a1b1} and \ref{a1b1half} gives
$$D(\textbf{a};\tau) \ll e^{\frac{\pi}{2}\sqrt{\frac{n}{12}}\left(-P_0\left(\frac{h}{p}\right)-6P_1\left(\frac{h}{p}\right) -P_2\left(\frac{h}{p}\right)+P_3\left(\frac{h}{p}\right)+P_4\left(\frac{h}{p}\right)\right)}. $$ Therefore,
$$G(\tau)D(\textbf{a};\tau)\ll e^{\frac{\pi}{2}\sqrt{\frac{n}{12}}F\left(\textbf{a};\frac{p}{h}\right)}.$$ Applying Lemma \ref{Fboundedby2} proves the claim.
\end{proof}

\subsection{Integration and proof of Theorem \ref{thebigboyforaandb}  for $a(n)$}\label{integrationsec}
We follow the approach of \cite{BringDousseLoveMahl} by approximating our integral with Bessel functions. We take the standard counter-clockwise path around the origin $\gamma:= \{e^{-2\pi i x}: x\in (-\frac{1}{2}, \frac{1}{2}] \}$. By Cauchy's theorem, we have
\begin{align*}
    a(n) = \int_{\gamma}R^{(3)}_3(q)q^{-n}dx = I_1 + I_2,
\end{align*}
where 
\begin{align*}
    I_1 &:= \int_{|x-\frac{1}{2}|<My}R^{(3)}_3(e^{2\pi i\tau})e^{-2\pi i n\tau}dx,\\
    I_2 &:= \int_{|x-\frac{1}{2}|> My}R^{(3)}_3(e^{2\pi i\tau})e^{-2\pi i n\tau}dx.
\end{align*}
Due to the bound in Theorem \ref{boundforeverythingfarfromminus1}, $I_1$ will be our main term. Applying Theorem \ref{mainterm}, we have that
\begin{equation*}
\begin{split}
     I_1 = \int_{|x-\frac{1}{2}|<My}T(\tau)e^{-2\pi i n\tau}dx &= \frac{e^{\frac{\pi i}{4}}\sqrt{3}}{6}\int_{|x-\frac{1}{2}|<My}\frac{Q_0^{-\frac{1}{192}}}{\sqrt{\tau-\frac{1}{2}}}\left(1+ O(e^{\pi\sqrt{\frac{n}{12}}})\right)e^{-2\pi i \tau n}dx\\
     &=\frac{e^{\frac{\pi i}{4}}\sqrt{3}}{6}\int_{|x-\frac{1}{2}|<My}\frac{e^{\frac{\pi i}{96(\tau-\frac{1}{2})}}}{\sqrt{\tau-\frac{1}{2}}}e^{-2\pi i \tau n}dx + E_1,
\end{split}
\end{equation*}
where $E_1\ll n^{\frac{1}{4}}$. Dealing with the remaining integral, we use the substitution $w := \tau-\frac{1}{2}$, and then $w = ivy$ (with $v= \frac{1}{\sqrt{192n}}$) to obtain,
\begin{align}\label{finalsubinscriptI}
\frac{e^{\frac{\pi i}{4}}\sqrt{3}}{6}\int_{|x-\frac{1}{2}|<My}\frac{e^{\frac{\pi i}{96(\tau-\frac{1}{2})}}}{\sqrt{\tau-\frac{1}{2}}}e^{-2\pi i \tau n}dx= -\frac{i(-1)^n\sqrt{3} \sqrt{y}}{6}\mathbb{I},
\end{align}
where we define
\begin{align*}
    \mathbb{I}:= \int^{1+iM}_{1-iM}\frac{e^{u\left(v+\frac{1}{v}\right)}}{\sqrt{v}}dv
\end{align*}
and $u:= \frac{\sqrt{3n}\pi}{12}$. Lemma 7 of \cite{BringDousseLoveMahl} gives the asymptotic expansion for $\mathbb{I}$ for such a $u$:
\begin{align*}\label{asymptforscriptI}
    \mathbb{I} = i\frac{\sqrt{12}}{3^{1/4}n^{1/4}}e^{\pi \sqrt{\frac{n}{12}}} + O\left(\frac{e^{\pi \sqrt{\frac{n}{12}}}}{n^{3/4}}\right).
\end{align*}
Subbing back into Eq. \eqref{finalsubinscriptI} and sending $n \to \infty$ gives
\begin{align*}
    I_1 \sim  a(n) \sim (-1)^n\frac{\sqrt{6}}{12\sqrt{n}}e^{\pi \sqrt{\frac{n}{12}}}. 
\end{align*}

\section{The $b(n)$}\label{thebnsec}
We now turn our attention to $R^{(3)}_1(q)=: \sum_{n\geq0} b(n)q^n$. To begin, we note that the $b(n)$ form a weakly increasing sequence, which is apparent from the definition of $\nu(q)$.
\begin{Lemma}\label{weaklyincreasinglemma}
Let $b(n)$ denote the $n^{th}$ Fourier coefficient of the function $\nu(-q)$. Then the sequence $\{b(n)\}^{\infty}_{n=0}$ is weakly increasing and no $b(n) <0$.
\end{Lemma}


The following Tauberian Theorem allows us to capture the growth of the $b(n)$ by only computing an estimate for the growth of $R^{(3)}_1$ in an angular region around $\tau = 0$. The original theorem is due to Ingham \cite{Ingham}, but we state it in a more modern form taking into account some additional technicalities regarding the growth of functions in angular regions around the origin.   

\begin{Theorem}[See Theorem 1.1 of \cite{TauberianSchaffer} with $\alpha =0$]\label{Tauberian}
Let $c(n)$ denote the coefficients of a power series $C(q):= \sum^{\infty}_{n=0}c(n)q^n$ with radius of convergence equal to $1$. Define $z:= x+iy \in \mathbb{C}.$ If the $c(n)$ are non-negative, are weakly increasing, and we have as $t\to 0^+$ that
\begin{align*}
    C(e^{-t}) \sim \lambda t^{\alpha}e^{\frac{A}{t}},
\end{align*}
and if for each $M>0$ such that $|y|\leq M|x|$
\begin{align*}
   C(e^{-z}) \ll   |z|^{\alpha}e^{\frac{A}{|z|}}
\end{align*}
with $A>0$, then as $n\to \infty$
\begin{align*}
    c(n)\sim \frac{\lambda A^{\frac{\alpha}{2}+\frac{1}{4}}}{2\sqrt{\pi}n^{\frac{\alpha}{2}+\frac{3}{4}}}e^{2\sqrt{An}}.
\end{align*}
\end{Theorem}
\begin{Remark}
We will show the bound in Theorem \ref{Tauberian} for $R^{(3)}_1$ as $\tau \to 0$ with $\tau \in \mathbb{H}$, which is sufficient to show the bound for general $z$ since we can define an even extension of $R^{(3)}_1$ into the lower half plane to get a function on all of $\mathbb{C}$.
\end{Remark}

\subsection{Growth near $\tau = 0$.}
We focus on the Appell sum $\mu\left(\frac{5}{12}, \frac{1}{4}; -\frac{1}{12\tau}\right)$ appearing in Eq. \eqref{muestimate}:
\begin{equation}\label{subintothissum}
 \begin{split}
    &\mu\left(\frac{5}{12}, \frac{1}{4}; -\frac{1}{12\tau}\right)= \frac{e^{\frac{5\pi i}{12}}}{\vartheta\left(\frac{1}{4}; -\frac{1}{12\tau}\right)}\sum_{n\in\mathbb{Z}}\frac{(-1)^nq^{\frac{n(n+1)}{24}}_0e^{\frac{n\pi i}{2}}}{1-e^{\frac{5\pi i}{6}}q^{\frac{n}{12}}_0}\\
    &= \frac{e^{\frac{5\pi i}{12}}}{\vartheta\left(\frac{1}{4}; -\frac{1}{12\tau}\right)}\Bigg\{\frac{1}{1-e^{\frac{5\pi i}{6}}}   + \sum_{n>0}\left(\frac{(-1)^nq^{\frac{n(n+1)}{24}}_0 e^{\frac{\pi i n}{2}}}{1-e^{\frac{5\pi i}{6}}q^{\frac{n}{12}}_0}+ \frac{(-1)^nq^{\frac{n(n-1)}{24}}_0 e^{-\frac{n\pi i}{2}} }{1-e^{\frac{5\pi i}{6}}q_0^{\frac{-n}{12}}}\right)\Bigg\}.
    \end{split}
\end{equation}
The last line follows by splitting the sum into $n<0$ and $n>0$, and then swapping $n\mapsto -n$ in the sum over $n<0$. We then have,
\begin{align*}
    &\sum_{n>0}\left(\frac{(-1)^nq^{\frac{n(n+1)}{24}}_0 e^{\frac{\pi i n}{2}}}{1-e^{\frac{5\pi i}{6}}q^{\frac{n}{12}}_0}+ \frac{(-1)^nq^{\frac{n(n-1)}{24}}_0 e^{-\frac{n\pi i}{2}} }{1-e^{\frac{5\pi i}{6}}q_0^{\frac{-n}{12}}}\right) = O(q^{\frac{1}{12}}_0).
\end{align*}
We sub this back into Eq. \eqref{subintothissum} to obtain
\begin{align}
\begin{split}\label{constantmuest}
    \mu\left(\frac{5}{12}, \frac{1}{4}; -\frac{1}{12\tau}\right)= \frac{-1}{2i \textnormal{sin}\left(\frac{5\pi}{12}\right)\vartheta\left(\frac{1}{4}; -\frac{1}{12\tau}\right)}\left(1+O(q^{\frac{1}{12}}_0)\right).
\end{split}
\end{align}
We can use the triple product formula to deal with the $\vartheta$-function to find as $\tau \to 0$ that

\begin{align}
\begin{split}\label{fourthperiodthetaest}
    \vartheta\left(\frac{1}{4}; -\frac{1}{12\tau}\right) &= -ie^{-\frac{\pi i}{4}}q^{\frac{1}{96}}_0(q^{\frac{1}{12}}_0;q^{\frac{1}{12}}_0)_{\infty}(i;q^{\frac{1}{12}}_0)_{\infty}(-iq^{\frac{1}{12}}_0;q^{\frac{1}{12}}_0)_{\infty}\\
    &\sim -2\textnormal{sin}\left(\frac{\pi}{4}\right)q^{\frac{1}{96}}_0.\\
\end{split}
\end{align}
Subbing Eqs. \eqref{constantmuest} and \eqref{fourthperiodthetaest}  into \eqref{muestimate} and using Lemma \ref{hlemma}, we have as $\tau \to 0$
\begin{align*}\label{mymu}
    \mu(5\tau,3\tau;12\tau)&=-q^2\frac{\mu\left(\frac{5}{12}, \frac{1}{4}; -\frac{1}{12\tau}\right)}{\sqrt{-12i\tau}} + \frac{h(2\tau;12\tau)}{2i}\\
    &\sim -\frac{q^{-\frac{1}{96}}_0}{4 i \textnormal{sin}\left(\frac{\pi}{4}\right)\textnormal{sin}\left(\frac{5\pi}{12}\right)\sqrt{-12i\tau}}.
\end{align*}
 Therefore, we can state the following.

\begin{Theorem}\label{muestimateneartau0}
As $\tau \to 0$ within a cone, we have the estimate
\begin{align*}
     \mu(5\tau,3\tau;12\tau)\sim -\frac{e^{\frac{\pi i}{48\tau}}}{4 i \textnormal{sin}\left(\frac{\pi}{4}\right)\textnormal{sin}\left(\frac{5\pi}{12}\right)\sqrt{-12i\tau}}.
\end{align*}
\end{Theorem}
We now show that the $\eta$-product 
\begin{align*}
     \frac{\eta(\tau + \frac{1}{2})\eta(3\tau+\frac{1}{2})\eta(12\tau)}{\eta(2\tau)\eta(6\tau)},
\end{align*}
has similar growth near $\tau =0$. To get the behavior near $0$ of the eta function involving the $\frac{1}{2}$ shift, we can proceed as we did in the proof of Lemma \ref{halfperiosthetaneartauhalf}. Define the transformation
\begin{align*}
    A:=\begin{pmatrix} 1&0\\2&1\end{pmatrix},
\end{align*}
with $w := -\frac{1}{4\tau}-\frac{1}{2}$ and then send $\tau \to 0$. We have that
\begin{equation}\label{fulletatrans}
    \eta(A\tau) = \epsilon(A)(2\tau+1)^{\frac{1}{2}}\eta(\tau) = e^{-\frac{\pi i}{6}}(2\tau+1)^{\frac{1}{2}}\eta(\tau).
\end{equation}
Lemma \ref{thetatestimateprelim} and Eq. \eqref{fulletatrans} say that near 0,
\begin{align}
\begin{split}\label{Aetatrans}
     \eta(Aw)&= e^{-\frac{\pi i}{6}}(2w+1)^{\frac{1}{2}}\eta(w) =  e^{-\frac{\pi i}{6}}(2w+1)^{\frac{1}{2}} e^{\frac{\pi iw}{12}}(e^{2\pi i w}; e^{2\pi i w})_{\infty} \\
     &\sim  e^{-\frac{5\pi i}{24}}\left(-\frac{1}{2\tau}\right)^{\frac{1}{2}} e^{-\frac{\pi i}{48\tau}} = \frac{ie^{-\frac{5\pi i}{24}}}{\sqrt{2\tau}}q^{\frac{1}{96}}_0.
\end{split}
\end{align}
Therefore as $\tau \to 0$ using Eq. \eqref{Aetatrans},
\begin{equation}\label{eta1plushalf}
    \eta\left(\tau + \frac{1}{2}\right) \sim \eta\left(Aw\right) \sim \frac{ie^{-\frac{5\pi i}{24}}}{\sqrt{2\tau}}q^{\frac{1}{96}}_0,
\end{equation}
\begin{equation}\label{eta3plushalf}
     \eta\left(3\tau + \frac{1}{2}\right) \sim \eta\left(A(3w)\right) \sim \frac{ie^{-\frac{5\pi i}{24}}}{\sqrt{6\tau}}q^{\frac{1}{288}}_0.
\end{equation}
The other $\eta$-products satisfy the estimates near zero directly from Lemma \ref{thetatestimateprelim} using the substitutions $\tau \mapsto 2\tau$, $\tau \mapsto 6\tau$, and $\tau \mapsto 12\tau$ respectively,
we have as $\tau \to 0$:
\begin{align}\label{etaprodnearzeroeq}
    \frac{\eta(\tau + \frac{1}{2})\eta(3\tau+\frac{1}{2})\eta(12\tau)}{\eta(2\tau)\eta(6\tau)} \sim i\frac{e^{-\frac{5\pi i}{12}}}{\sqrt{-12i\tau}}q^{-\frac{1}{96}}_0.
\end{align}
\vspace{5mm}
\subsection{Proof of the estimate for the $b(n)$}
We now prove the main theorem for the $b(n)$.
\begin{Theorem}\label{finalanswer} Let $b(n)$ denote the coefficients of $\nu(-q)$. Then as $n\to \infty$,
\begin{align*}
    b(n) \sim \left(\frac{1}{2\textnormal{sin}\left(\frac{\pi}{4}\right)\textnormal{sin}\left(\frac{5\pi}{12}\right)} +1\right)\frac{e^{\pi\sqrt{\frac{n}{6}}}}{\sqrt{24n}}.
\end{align*}
\end{Theorem}

\proof
Combining  Theorem \ref{muestimateneartau0} and Eq. \eqref{etaprodnearzeroeq}, we have as $\tau \to 0$ that

\begin{align}
\begin{split}\label{totalgrowthnearo}
   \nu(-q) =  R^3_1(q) &\sim 2i\frac{q_0^{-\frac{1}{96}}}{4 i \textnormal{sin}\left(\frac{\pi}{4}\right)\textnormal{sin}\left(\frac{5\pi}{12}\right)\sqrt{-12i\tau}} + \frac{1}{\sqrt{-12i\tau}}q^{\frac{1}{96}}_0\\
    &=\frac{e^{\frac{\pi i}{48\tau}}}{\sqrt{-12i\tau}}\Big(\frac{1}{2 \textnormal{sin}\left(\frac{\pi}{4}\right)\textnormal{sin}\left(\frac{5\pi}{12}\right)}+1\Big),
\end{split}
\end{align}
where $q^{-\frac{1}{2}}\to 1$ in the limit $\tau \to 0$. Making the substitution $ \tau := \frac{it}{2\pi}$, we have that as $t \to 0^+$ that
\begin{align*}
    R^3_1(e^{-t}) =\left(\frac{1}{2\textnormal{sin}\left(\frac{\pi}{4}\right)\textnormal{sin}\left(\frac{5\pi}{12}\right)} +1\right)\sqrt{\frac{\pi}{6}} \frac{e^{\frac{\pi^2}{24t}}}{ \sqrt{t}}.
\end{align*}
The bound for the complex variable, $z$, in Theorem \ref{Tauberian} is trivially satisfied by combining the estimates in  Theorem \ref{muestimateneartau0} and Eq. \eqref{etaprodnearzeroeq}.
\par
Define $A:= \frac{\pi^2}{24}$ and $\lambda:= \left(\frac{1}{2\textnormal{sin}\left(\frac{\pi}{4}\right)\textnormal{sin}\left(\frac{5\pi}{12}\right)} +1\right)\sqrt{\frac{\pi}{6}}$. By Theorem \ref{Tauberian} with $\alpha = \frac{1}{2}$, we have that
\begin{align*}
      b(n)&\sim \frac{\lambda}{2\sqrt{\pi}n^{\frac{1}{2}}}e^{2\sqrt{An}} = \left(\frac{1}{2\textnormal{sin}\left(\frac{\pi}{4}\right)\textnormal{sin}\left(\frac{5\pi}{12}\right)} +1\right)\sqrt{\frac{\pi}{6}}\frac{1}{2\sqrt{\pi}\sqrt{n}}e^{\sqrt{\frac{n}{6}}}\\
     & \sim \left(\frac{1}{2\textnormal{sin}\left(\frac{\pi}{4}\right)\textnormal{sin}\left(\frac{5\pi}{12}\right)} +1\right)\frac{e^{\pi\sqrt{\frac{n}{6}}}}{\sqrt{24n}},
\end{align*}
which shows the claim.\qed

\section{Conclusions and future work}\label{conclusion}
This work studied the estimates for the Fourier coefficients $a(n)$ and $b(n)$ for the base cases of the $R^{(k)}_1$ and $R^{(k)}_3$, respectively. Both of this families are mock theta families that were derived from Bailey chains in \cite{lovejoybaileypairs}. We expect that generalizing to $k>3$ should be doable by brute-force methods given that many of the features of the $\theta_{1,p}$ function that we encountered with $p=4$ in this work, generalize for $p>4$. Much of this can be seen in Hickerson and Mortenson's original work \cite{splittingmort}. This includes the symmetry of the indefinite quadratic form $Q(r,s)$ that appears in the exponent of $q$ in the sum of $\theta_{1,p}$, which will allow for simpler expressions for the $\theta_{1,p}$ like we found in this work. Albeit possible to do without, it would be nice to find more elegant methods for dealing with the asymptotics for these families of Bailey mock theta functions. Based on numerical checks of the Fourier coefficients, we expect that the $R^{(k)}_1$  have weakly increasing coefficients for $k>3$.
\begin{Conjecture}
The $R^{(k)}_1$ have weakly increasing coefficients for all $k\geq 3$. 
\end{Conjecture}
Proving this by purely combinatorial means seems difficult, but possible using the many representations of $R^{(k)}_1$ given by Lovejoy and Osburn in \cite{lovejoybaileypairs}. One such way may involve appealing to some generalized $q$-binomial theorems and formulae for Gauss sums, like those posed in \cite{obscuregausssum}. Such an idea seems reasonable since the $B_k$ in the definition of $R^{(1)}_k$ can be expressed as weighted sums of Gaussian polynomials. For example,
\begin{align*}
    B_4(n_4,n_3,n_2,n_1; q) &= (-1)^{n_1}q^{\frac{n_3(n_3+1)}{2}+n_2+2n_1}\;\frac{(-q)_{n_3}(-q)_{2n_2}(-q^2;q^2)_{2n_1}}{(q)_{n_4-n_3}(q^2;q^2)_{n_3-n_2}(q^4;q^4)_{n_2-n_1}(q^8;q^8)_{n_1}}\\
    \\
     &=  (-1)^{n_1}q^{\frac{n_3(n_3+1)}{2}+n_2+2n_1}\begin{bmatrix}n_4\\n_3\end{bmatrix}_q\begin{bmatrix}n_3\\n_2\end{bmatrix}_{q^2}\begin{bmatrix}n_2\\n_1\end{bmatrix}_{q^4}\frac{(-q;q^2)_{n_2}(-q^2;q^4)_{n_1}}{(q)_{n_4}},
\end{align*}
where $\begin{bmatrix}m\\n\end{bmatrix}_q:= \frac{(q)_m}{(q)_{m-n}(q)_n}$ is the Gaussian $q$-binomial coefficient.


\printbibliography

\end{document}